\documentclass{amsart} 
\usepackage{latexsym,amssymb,graphicx,amscd,amsmath}

\newtheorem{thm}{Theorem}[section]
\newtheorem{lem}[thm]{Lemma}
\newtheorem{prop}[thm]{Proposition}
\newtheorem{cor}[thm]{Corollary} 
\newtheorem{de}[thm]{Definition} 
\newtheorem{rem}[thm]{Remark}
\newtheorem{ex}[thm]{Example}

\newcommand{\BZ}{{\mathbb{Z}}}
\newcommand{\BR}{{\mathbb{R}}}
\newcommand{\BQ}{{\mathbb{Q}}}
\newcommand{\BO}{{\mathcal{O}_p}}

\newcommand{\BOplus}{{\mathcal{O}_p^+}}
\newcommand{\BS}{{\mathcal{S}_p}}
\newcommand{\BSplus}{{\mathcal{S}^+_p}}

\newcommand{\BH}{{\mathcal{H}}}

\newcommand{\BF}{{\mathbb{F}}}

\newcommand{\Si}{{\Sigma}}

\newcommand{\Sip}{{\Sigma^{\diamond}}}
\DeclareMathOperator{\Heeg}{M_{\BH}}
\DeclareMathOperator{\Signature}{Sign}

\newcommand{\I}{{\mathrm I}}

\newcommand{\cC}{{\mathcal{C}}}

\newcommand{\la}{{\lambda}}

\DeclareMathOperator{\rad}{rad}
 
\DeclareMathOperator{\genus}{genus}

\DeclareMathOperator{\Lk}{Lk}
\DeclareMathOperator{\sgn}{sgn}
\DeclareMathOperator{\Int}{Int}

\DeclareMathOperator{\Id}{Id}

\DeclareMathOperator{\Sp}{Sp}
\DeclareMathOperator{\rank}{rank}

\begin{document}
\title[Maslov index, Mapping Class Groups and TQFT]{Maslov index, Lagrangians,
Mapping Class Groups and TQFT}

\author{Patrick M. Gilmer}
\address{Department of Mathematics\\
Louisiana State University\\
Baton Rouge, LA 70803\\
USA}
\email{gilmer@math.lsu.edu}
\thanks{The first author was partially supported by  NSF-DMS-0604580,  NSF-DMS-0905736}
\urladdr{www.math.lsu.edu/\textasciitilde gilmer/}

\author{Gregor Masbaum}
\address{Institut de Math{\'e}matiques de Jussieu (UMR 7586  
CNRS)\\  
Case 247\\
4 pl. Jussieu\\
75252 Paris Cedex 5\\
FRANCE }
\email{masbaum@math.jussieu.fr}
\urladdr{www.math.jussieu.fr/\textasciitilde masbaum/}

\date{July 3, 2011}

\begin{abstract} Given a mapping class $f$ of an oriented
  surface $\Si$ and a
  lagrangian $\la$ in the first homology of $\Si$, we define an
  integer  
$n_\la(f)$.  
We use 
$n_\la(f)\pmod 4$
to describe 
  a
   universal central extension of the
  mapping class group of $\Si$ as an index-four subgroup of the
  extension constructed from the Maslov index of triples of lagrangian subspaces in the homology of
  the surface. We give two 
descriptions of this subgroup.  
One is 
  topological using surgery, the other is homological and builds on
  work of  Turaev and work of Walker. Some applications to TQFT are
  discussed. They are based on the fact that our construction allows one to precisely describe how the phase
  factors that arise in the skein theory approach to TQFT-representations of the mapping class group depend on the choice of a
  lagrangian on the surface.
\end{abstract}

\maketitle
\tableofcontents
 
\section{Introduction} \label{sec.intro}

The mapping class group  $\Gamma_{g}$ of a surface of genus $g$ has a
long history in low-dimensional topology. In this paper, we are
concerned with central extensions of 
$\Gamma_{g}$, which have proved to be important in TQFT.  It follows from
Harer's work \cite{H} that $\Gamma_{g}$ has  a universal central
extension by $\BZ$, for $g \ge 5$ (later works improve this to $g \ge
4$). The 
cohomology class  of 
 such an 
extension is a generator of  
$H^2(\Gamma_g;\BZ)$ (this group is isomorphic to $\BZ$ for $g\geq 3$). 
One way to obtain explicit $2$-cocycles representing cohomology
classes of central extensions of $\Gamma_g$ is to pull back cocycles of the
symplectic group $\Sp(g,\BZ)$ via the map  
$\Gamma_g \rightarrow \Sp(g,\BZ)$
which sends a mapping class $f\in \Gamma_{g}$ to
the induced map  
on the homology of the surface. 
The most prominent such $2$-cocycle among topologists is probably the signature
cocycle for $\Sp(g,\BZ)$, defined by Meyer \cite{Me} using  
signatures of certain 4-manifolds which fiber over a disk with two
holes. We will use $\tau$ to denote the pull-back of Meyer's cocycle
to the mapping class group $\Gamma_{g}$.
Meyer's work  implies that the cohomology class $[\tau]$  is divisible
by four, and the class $[\tau]/4$ is a generator of 
$H^2(\Gamma_g;\BZ)$.  
However, $\tau$ itself
is not divisible by $4$, and Meyer did not
give an explicit $\BZ$-valued cocycle representing $[\tau]/4$. This
was done by 
Turaev   \cite{T2,T3}, who 
had independently studied the signature cocycle from a different 
point of view. Turaev 
showed how to
modify $\tau$ by the coboundary of a certain explicit $1$-cochain to find a cocycle which is divisible by
four. Thus, Turaev's work gives an explicit cocycle for 
a
 universal
central extension of $\Gamma_{g}$.

Renewed interest in these questions was sparked by Atiyah \cite{A},
who pointed out that the signature cocycle was closely
related to the problem of resolving anomalies in TQFT. Anomalies are
responsible for the fact that TQFT-representations of mapping class
groups are often only projective representations. Resolving the
anomalies means replacing these projective representations by
linear representations of appropriate central extensions of the
mapping class group. 
In \cite{A}, Atiyah 
suggested the notion of 
$2$-framings 
to resolve anomalies.  Blanchet,
Habegger, Masbaum and Vogel  \cite{BHMV2} used the notion of
$p_1$-structures to resolve anomalies in their construction of TQFT's
from the skein theory of the Kauffman bracket. The 
projective factors arising in the 
skein-theoretical
construction of TQFT were computed explicitly in Masbaum-Roberts
\cite{MR}. 

The central extensions of $\Gamma_g$ considered in the present paper
are constructed 
using yet another approach to resolving
anomalies which was 
pioneered by Walker \cite{W}, and further developed by Turaev
\cite{T}.  
 For an early use of this approach, see \cite{JEA}.
As far as the mapping class group is concerned, this method depends on fixing a lagrangian subspace  $\lambda$ 
 of the first rational homology of the surface. One then uses the Maslov index
 of 
triples 
of lagrangian subspaces to define a central extension of
 $\Gamma_g$.  
Let us denote this extension by $\widetilde \Gamma_g$.  
 The group $\widetilde \Gamma_g$ is thus given explicitly as
the set of pairs $\{(f,n)| f \in \Gamma_g, n\in \BZ\}$, with
multiplication defined by a certain cocycle $m_\lambda$ which we call
the Maslov cocycle. This cocycle is
also known as the Shale-Weil cocycle, which is discussed for instance
in \cite{LV}. 
In contrast with the signature cocycle $\tau$, the Maslov cocycle 
depends on the chosen lagrangian $\lambda$. But it turns out that in
cohomology, one has $[m_\la] =-[\tau]$. Thus the class $[m_\la]/4$
corresponds to an index-four  subgroup 
of $\widetilde \Gamma_g$ which we denote by  ${\widetilde \Gamma_g}^{++}$. If
$g \ge 4$, ${\widetilde \Gamma_g}^{++}$
is a universal central extension  
of 
$\Gamma_g.$ The main aim of the
present paper is to explain how one can get one's hands on explicit elements of
this group ${\widetilde \Gamma_g}^{++}$, and to understand the 
 role 
 played by the chosen lagrangian $\la$ in this description.

Let us briefly describe the organization and main results of this
paper.  We find it convenient to denote the extended mapping
class group $\widetilde \Gamma_g$ by
$\widetilde\Gamma(\Si)$,  where 
$\Si$ stands for the `extended' surface 
consisting of a surface together with a fixed 
lagrangian (see the beginning of
Section~\ref{sec3} for more details). Similarly, we will denote
${\widetilde \Gamma_g}^{++}$ by $\widetilde\Gamma(\Si)^{++}$. 
In Sections~\ref{sec2} and~\ref{sec3}, we review basic concepts about
Maslov index and the 
extended cobordism category and define the extended mapping class
group.   
The multiplication in $\widetilde \Gamma(\Si)$ is defined in formula
(\ref{comp}) in Section~\ref{sec3} (this formula is restated in terms of the
Maslov cocycle in formula (\ref{ml1}) in Section~\ref{sec5b}).

 In Section~\ref{newsec3}, we use
extended surgery to define certain specific lifts of Dehn twists to 
$\widetilde\Gamma(\Si)$ and prove a surgery formula computing, for any
word $\mathfrak w$ in Dehn twists, the product in
$\widetilde\Gamma(\Si)$ of the corresponding lifts. This formula is stated in
Theorem~\ref{new4.1}. It involves the signature of the linking matrix
of a framed link constructed from the  
word $\mathfrak w$ and the lagrangian $\la$.
Our construction here is somewhat similar to the
work of Roberts and one of us in
\cite{MR}, but the context is different, as there were no lagrangians
in \cite{MR}. Also, the framed link we are using is different from the
one used in \cite{MR}. The framed link used in \cite{MR} would be  
appropriate for our purposes only for words $\mathfrak w$ representing
the identity 
mapping class, but not in general.

In Section~\ref{newsec5}, we then define  $\widetilde \Gamma(\Si)^{++}$  
as the subgroup of ${\widetilde \Gamma(\Si)}$ 
generated by the above-mentioned lifts of Dehn twists, slightly
shifted (see Definition~\ref{ind4}).  The fact that $\widetilde
  \Gamma(\Si)^{++}$ has index four in 
${\widetilde \Gamma(\Si)}$ is not obvious from this definition. This fact 
will follow from a second, purely algebraic 
description of $\widetilde \Gamma(\Si)^{++}$, which we state in
Section~\ref{sec5} and prove in Section~\ref{sec5b}. We define an
integer $n_\la(f)$ for any mapping class $f$ and lagrangian $\lambda$
and show in Theorems~\ref{hom} and
its Corollary \ref{hom-ind} 
that $\widetilde \Gamma(\Si)^{++}$ is the subset of
$\widetilde \Gamma(\Si)$ given by the $(f,n)$ with $n\equiv n_\la(f)
\pmod 4$. 
Our
formula for $n_\la(f)$ uses Turaev's $1$-cochain from \cite{T2,T3}, but adds
to it a term which explicitly depends on the lagrangian $\la$. It is
remarkable that Turaev's cochain  is defined using a certain
non-symmetric bilinear form depending only on $f$, while our
additional term is the signature of this same form restricted to a subspace on
which the form is symmetric (but the subspace depends on the lagrangian). The
proof of  
Theorem~\ref{hom}
uses 
a formula of
Walker \cite[p.~124]{W} relating the signature cocycle to the Maslov
cocycle. We remark that Walker's formula is in an unfinished manuscript, which does
not claim to get the signs right.    
We state a
version of 
his formula, 
in terms of our definitions and conventions,  as
Theorem  \ref{ctd}, and give a detailed version of the proof Walker
outlines. Also, Turaev  defined his
version of the signature
cocycle in a purely algebraic fashion,  and he did not give the precise
relationship with Meyer's definition. In fact, Turaev's cocycle turns out
to be equal 
to $-\tau$, see Proposition~\ref{fgint}. Since we are
proving a congruence modulo four (and not just modulo two), getting
the signs right is important for us, so we have tried to
deal with these sign issues in some detail. 

In Section~\ref{newsec7}, we discuss the relationship of our
index four subgroup $\widetilde \Gamma(\Si)^{++}$ of $\widetilde
\Gamma(\Si)$ with the index two subgroup  $\widetilde \Gamma(\Si)^{+}$
constructed by one of us in \cite{G}. (It is this relationship which
motivated the superscript $++$ in our notation for  $\widetilde
\Gamma(\Si)^{++}$.)  
 
The preceding results all extend to the mapping class group of a
surface with boundary. The (small) modifications required to do so are
explained in Section~\ref{sec6}. We also explain briefly in
Section~\ref{sec7} how one sees that 
$\widetilde \Gamma(\Si)^{++}$ 
is    a universal central extension in genus at least four.

The remainder of the paper is devoted to applications of our
results to TQFT. As already said, we use Walker's \cite{W} and
Turaev's \cite{T} approach to TQFT, where one
 consider surfaces equipped with the extra structure of a lagrangian
 subspace of their first homology, and 3-manifolds equipped with an
 integer weight. These are called extended manifolds, and the
 resulting extended cobordism category is used to resolve the anomalies that
 arise in TQFT.   We believe that the skein theory approach of
 \cite{BHMV2} modified by substituting  extended manifolds for
 manifolds with $p_1$-structures is the most concrete and computable
 approach to the TQFTs associated to $SU(2)$, and  $SO(3).$ The reason for this precision is
 that a lagrangian subspace may be specified algebraically while a
 $p_1$-structure is harder to specify. In
 Section~\ref{sec8}, we explain how this works  in practice for the
 mapping class group representations. 
See for instance, Theorem
 \ref{8.2}, where we state precisely how the action of an element
 $(f,n)$ of the extended mapping class group 
on the TQFT-module associated to the surface $\Si$
 depends on the chosen lagrangian $\la$. We then use this to do some
 explicit computations (see Proposition~\ref{8.7}) which were used
 in~\cite{GM1}. The beginning of Section~\ref{sec8} is written so as
 to provide a further and more detailed introduction to the
 TQFT-aspects of our results.
 
 In the last section, we briefly consider the integral TQFT
 that we have been studying in \cite{G,GM,GM1} using the precision
 afforded by using extended manifolds. In Corollary \ref{modular}, we
 show that the representations coming from integral TQFT when
 restricted  
to 
$\widetilde \Gamma(\Si)^{++}$
induce modular representations of the ordinary mapping class group.
This was one of our motivations for studying the index  four
    subgroup  $\widetilde \Gamma(\Si)^{++}$.

\vskip 8pt

\noindent{\em Acknowledgments:}
 We thank the referee for his 
comments
which helped us to improve the organization of the paper.

\section{Maslov index, extended manifolds and extended surgery}\label{sec2}
Extended surfaces and 3-manifolds were introduced by Walker \cite{W} and further developed by Turaev \cite{T}. We begin by briefly describing these
notions to fix our conventions, and sketch the background.

Let $V$ be a rational vector space with a nonsingular skew-symmetric form  $\cdot: V \times V \rightarrow \BQ.$
A subspace $\lambda \subset V$ is called lagrangian
if $\lambda= \lambda^\perp$ where   $\lambda^\perp=\{ x\in 
V  
\, | \,
x \cdot    y=0, \ \forall y \in \lambda\}.$
It is easy  to see that $\lambda$ is lagrangian if and only if  $\lambda \subset \lambda^\perp$ and 
$\lambda$ has dimension $(1/2) \dim(V)$.
Recall  the Maslov index of an ordered triple of 
lagrangians $\lambda_1, \lambda_2$,  $\lambda_3$ in $V$. The Maslov index 
$\mu(\lambda_1, \lambda_2,  \lambda_3) \in \BZ$ is defined to be the
signature of the bilinear symmetric form $\odot$ on 
$(\lambda_1 + \lambda_2)\cap \lambda_3$ defined by $(a_1+a_2) \odot
(b_1+b_2) = a_2 \cdot b_1.$  
(Here $a_i, b_i \in \lambda_i$ for $i=1,2$,
and $a_1+a_2, b_1+b_2\in \lambda_3$.)   
We will need the following well-known
property of $\mu(\lambda_1, \lambda_2,  \lambda_3)$.

\begin{lem}\label{2lag} The Maslov index changes sign under  an odd
  permutation of the three lagrangians. In particular, $\mu(\lambda_1,
  \lambda_2,  \lambda_3)=0$   if two of the lagrangians are the same.
 \end{lem}

Recall that the first homology of a closed oriented 2-dimensional manifold $\Si$ has a skew-symmetric intersection form 
$\cdot: H_1(\Si; \BQ) \times H_1(\Si; \BQ) \rightarrow \BQ.$  By a lagrangian of $\Si$, we mean a lagrangian for  $H_1(\Si; \BQ)$ with this pairing.

 An extended surface $\Si$ is a closed oriented 2-dimensional manifold  equipped with a
   lagrangian subspace $\lambda (\Si)\subset H_1(\Si; \BQ)$. 
It is clear how to take the disjoint union of extended surfaces.

An extended 3-manifold $M$  is a compact oriented 3-dimensional manifold equipped with a weight $w(M) \in \BZ$, and whose oriented boundary $\partial M$ has been given the structure of an extended surface with a lagrangian $\lambda(\partial M)$. 
In this case $\partial
M$ also has a lagrangian  given by  $\text{kernel} (i_{*}), $ where $i:\partial M \rightarrow M$ is the inclusion. We denote this lagrangian by 
$\lambda_M(\partial M)$. We insist that $\lambda(\partial M)$ could be
chosen arbitrarily and will usually be different from
$\lambda_M(\partial M)$.

If $M$ is an extended $3$-manifold and $\Si$ is a connected component
of $\partial M$, then $\lambda(\partial M) \cap H_1(\Si;\BQ)$ may or
may not be  a lagrangian for $\Si$. If it is a lagrangian for $\Si$, 
we may equip $\Si$ with this lagrangian and we will call $\Si$, so equipped,  a boundary surface of the extended 3-manifold $M$.

Extended $3$-manifolds can be glued along boundary surfaces. To
describe this `extended' gluing, we need one more notation. First, observe that if $\Si$ is a boundary 
surface of $M$, then 
$\Sip=\partial M \setminus  \Si$ 
is also a boundary surface, 
and $\partial M$ is the disjoint union of $\Si$ and 
$\Sip$ 
as 
extended surfaces. Now let  $i_\Si$ and $i_{\Sip}$ denote the
inclusions of $\Si$ and $\Sip$ into $M$, and define $\lambda_M(\Si)$
to be $i_{\Si}^{-1} \left(i_{\Sip}(\lambda(\Sip))\right).$ In other
words, we restrict the given lagrangian $\lambda(\partial M)$ to
$\Sip$, and then `transport' it over to $\Si$, using $M$. Note that
if $\Si$ is the whole boundary of $M$, so that $\Sip = \emptyset$,
this agrees with the earlier definition of $\lambda_M(\partial
M)$. As before, we insist that $\lambda_M(\Si)$ will in general be different from
$\lambda(\Si)$. 
 
Throughout this paper, we denote orientation reversal by an
overbar. If $\Si$ is an extended surface, 
$\overline{\Si}$ 
denotes the same
surface with opposite orientation and with the same 
 lagrangian $\lambda(\overline \Si)= \lambda(\Si)$.  
 If
$M$ is an extended $3$-manifold, 
$\overline M$  
denotes the same manifold
with opposite orientation and weight 
$w(\overline {M})=- w(M)$.

We can now spell out the gluing formula. Let $M$ and $M'$ be two
extended 3-manifolds and assume that  $\Si$ is a boundary surface of
$M$ and 
$\overline \Si$ 
is a boundary surface of $M'$.
Then we
 may glue $M$ and $M'$ (by the orientation reversing identity
map from $\Si$ to 
$\overline \Si$) 
 together to form a new extended
$3$-manifold  
$M\cup_\Si M'$.   
The weight of  
$M\cup_\Si M'$ 
is defined as 
\begin{equation} 
\label{weightunion}
w(M\cup_\Si M')= w(M)+w(M') - \mu_\Si\left(\lambda_M(\Si), \lambda(\Si), \lambda_{M'}(\overline \Si)\right).\end{equation}
We write $\mu_\Si$ to indicate that this Maslov index  is to be computed using the intersection form of $\Si$, rather than $\overline \Si.$ 
We note that $\lambda_{M'}(\overline \Sigma)$ is a lagrangian for both $\Si$ and $\overline \Si$ as the notion of lagrangian does not depend on the orientation of  the surface. The minus sign in the above formula is needed to make Lemma \ref{4intrep} hold.

We would get the same number computing:
\[w(M'\cup_{\overline \Si} M)= w(M')+w(M) - \mu_{\overline{ \Si}}\left(\lambda_{M'}(\overline  \Si), \lambda(\overline \Si), \lambda_M(\Si)\right),\]
as the intersection 
pairings differ by a sign 
 but an odd permutation of the lagrangians has been  introduced. 

Thus gluing of extended manifolds is `commutative'. In other words,
it does not matter whether we think we are gluing $M$ to $M'$ or $M'$ to
$M$. Gluing is also `associative', meaning that if we have a
collection of extended 3-manifolds that we wish to glue together along
boundary surfaces, it does not matter in what order we do the gluing.  This follows from the geometric interpretation
of weights in terms of signatures of associated 4-manifolds given by Walker, as well as by the more algebraic approach given in  Turaev's book.

We now wish to define  the notion of 
extended 
surgery to an extended manifold
$M$ along a framed knot $K$ in $M$.
The resulting extended manifold will be denoted by $M_K$. Its
underlying manifold is obtained by the usual surgery procedure: we use
the framing and the orientation 
of $M$ to identify  a closed tubular
neighborhood $\nu(K)$ of $K$  with $\overline{S^1} \times D^2$; we
then cut
out the tubular neighborhood, and replace it with $D^2 \times
S^1$. (Note that $\partial (\overline{S^1} \times D^2)= S^1 \times S^1= \partial(D^2 \times
S^1)$.) Now, to make $M_K$ into an extended manifold, we do the same
thing but use extended gluing, where the extended structure is as
follows: We give $M \setminus
{\Int}(\nu(K))$ the weight of $M$, the weight of $D^2 \times S^1$  is
zero, and we equip  $S^1 \times S^1$
with the lagrangian generated by the homology class of the meridian of
the knot $K$, {\em i.e.},  ${\text{pt}} \times S^1$. We remark that
this is a natural choice for the lagrangian, as with this choice the
result of extended gluing of $M \setminus
{\Int}(\nu(K))$ with $\nu(K)$ (equipped with zero
weight) is $M$ with its original weight. (This follows from
Lemma~\ref{2lag}.)

Note that since $K$ is a knot, we have $|w(M_K)-w(M)| \leq 1$, as the
contribution from the Maslov index to the weight of $M_K$ is computed from a symmetric bilinear form on a
space of dimension at most one.
If we have a framed link $L$ in $M$, we may do a sequence of such
extended surgeries  or perform the surgeries all at once, and we would
get the same result (by the above-mentioned `associativity' of gluing). The resulting extended manifold is denoted ${M_L}$
and is called extended surgery along $L$. 

If $L$ is a framed ordered oriented link in $S^3$, let $\sigma(L)=
b_+(L)-b_-(L)$, where $b_\pm(L)$ is the number of positive (negative) eigenvalues (counted with multiplicity) of
the {\em linking  matrix} of $L$, that is, the symmetric integral matrix whose off-diagonal entries are the linking numbers
of the components of $L$, and whose diagonal entries are the framings. The number
  $\sigma(L)$ is the signature of the linking  matrix of $L$ and should
  not be confused with what is usually called the signature of the
  {\em link} $L$ in
  knot theory. Changing the order or the orientation of $L$ does
not effect $\sigma(L)$, $b_+(L)$, or $b_-(L)$.

The 4-manifold interpretation of weights \cite{W} yields the following
basic fact.

 \begin{lem}\label{4intrep} If $S^3$ is equipped with weight $w(S^3) =0$, then $w({(S^3)_L})=\sigma(L)$.
  \end{lem}

\section{The central extension $\widetilde\Gamma(\Si)$ of the mapping
  class group $\Gamma(\Si)$} 
\label{sec3}

We will realize our central extensions of the mapping class group as
subgroups of a certain extended cobordism category $\cC$. 
The objects of $\cC$ are 
 extended surfaces.  A morphism in $\cC$ from $ \Si$
to $ \Si'$ is given by an extended cobordism, that is, an extended $3$-manifold $M$ whose boundary has been
partitioned into the disjoint union of two  boundary surfaces,
 one of
which is identified with $\Si$ 
 by an orientation reversing  
diffeomorphism, and the other
is
identified with ${\Si'}$ by an orientation preserving diffeomorphism.
We denote such a cobordism by $ M: \Si \rightsquigarrow \Si'.$ 
We refer to  $\Si$ as the source and $\Si'$ as  the target of the cobordism.
If we
also have another cobordism $ M': \Si' \rightsquigarrow  \Si''$, we
can form $ M' \circ  M:  \Si \rightsquigarrow  \Si''$ by extended gluing
$ M$ to $ M'$  along $ \Si'$.  
Thus, 
$ M' \circ  M$ means {\em first $M$, then $M'$.} This convention
is needed to make formula (\ref{comp}) below hold.

Two 
extended
cobordisms from $ \Si$ to $
\Si'$ are considered equivalent  
if they have the same weight and   
if there is an orientation preserving
diffeomorphism between them which preserves their boundary
identifications.   Composition of  
extended
cobordisms is associative (on
equivalence classes). Therefore we define the morphisms of $ \cC$ from
$ \Si$ to $ \Si'$ to be  equivalence classes of 
extended
cobordisms. However,
from now on we will treat equivalent cobordisms as if they are identical. 
When it should cause no confusion, we will act as if the boundary identifications of a cobordism are identity maps.

Sometimes we will need to discuss 
extended manifolds whose extended structure we have forgotten, then we
will denote them by $\underline{M}$, $\underline \Si$ {\em etc.}  Thus, forgetting
the extended structure will be denoted by an underbar. We have a
forgetful functor $\cC \rightarrow \underline {\cC}$, where
$\underline {\cC}$ denotes the usual cobordism category, with
composition given by the usual gluing.

We now set out to define the extended mapping class group  $\widetilde
\Gamma(\Si)$ of a  closed 
oriented surface equipped with a fixed
lagrangian $\lambda(\Si)$. 
{\em Here and whenever we  discuss a mapping class group of a 
 surface in this paper, we
assume that the surface is connected.}
First of all, we denote by $\Gamma(\Si)$ the ordinary mapping class group
of the underlying surface $\underline{\Si}$. (The group $\Gamma(\Si)$
should 
perhaps be
denoted by $\Gamma({\underline \Si})$, but we find this notation too
clumsy.) Thus,   $\Gamma(\Si)$  is
the group of isotopy classes of
orientation-preserving diffeomorphisms of $\underline{\Si}$. 
Abusing
notation, we will write $f$ for a diffeomorphism, and its isotopy
class. 

 If $f\in \Gamma(\Si)$ and $n \in \BZ$,  we let  ${C}(f,n)$ denote  the extended cobordism given by
 the mapping cylinder of $f$  with weight $n$,  where  both the source and
 target are 
$\Si$ 
equipped with the lagrangian $\lambda(\Si)$. We call
 ${C}(f,n)$ an extended mapping cylinder. It is a morphism of $\cC$. Its underlying cobordism is
 the usual mapping cylinder 
of $f$, that is, the
 cobordism formed from 
$ \I \times \underline{\Si} $ by  identifying  
$\overline{\{0\} \times \underline \Si} $ 
with 
 the 
  source
  surface 
$\underline{\Si}$ 
via the identity (which is in this case is orientation reversing)
and identifying 
$\{1\} \times \underline \Si$ 
with 
the 
target 
surface
 $\underline \Si$ via  $f$.

It follows from (\ref{weightunion}) that  
composition of extended mapping cylinders is given by 
  \begin{align} 
  {C}(g,n) \circ  {C}( f,m)&=  
  {C}(g \circ f, n+m- \mu\left(  
  f_* \lambda(\Si), \lambda(\Si), g_*^{-1}\lambda(\Si)
  \right) )
  \notag \\ \label{comp} 
  &=  {C}(g \circ f, n+m+ \mu\left( 
  \lambda(\Si),g _* \lambda(\Si),(g \circ f)_* \lambda(\Si)
  \right) ) 
  \end{align}

\begin{de}\label{comp3} (Walker) The extended mapping class group is   \begin{equation}
 \widetilde \Gamma(\Si)= \{{C}(f,n)
\, | \,
f \in \Gamma(
\Si
), \ n \in \BZ \}
 \notag
\end{equation} with multiplication given by (\ref{comp}). \end{de}

We have a short exact sequence of groups (see Remark~\ref{comp4} below):
\[
 \begin{CD}
0 @>>> \BZ @>>> \widetilde \Gamma(\Si)  @>>>  \Gamma({\Si} ) @>>>  
1    .
\end{CD} 
\]  

The map $\widetilde \Gamma(\Si)  \rightarrow
\Gamma({  \Si})$  is given by ${C}(f,n) \mapsto f$. This is a central extension. The
kernel is generated by $C( \Id_\Si,1) \in  \widetilde \Gamma(\Si)$. We
denote this central generator by $W$.

\begin{rem}\label{comp4}{\em  In Definition~\ref{comp3}, we
realize $\widetilde\Gamma({\Si})$  as a subset of the endomorphisms
of ${\Si}$ in the extended cobordism  category $
{\cC}$. But notice that the extended mapping cylinder ${C}(f,n)$
(which we view as an equivalence class of  morphisms in ${\cC}$)
determines $(f,n)\in \Gamma({  \Si})\times \BZ$, because of
the following fact: One has that $f=g$ in $\Gamma({  \Si})$ if and only if
the (ordinary) mapping cylinders of $f$ and $g$ 
are equivalent as morphisms
of $\underline \cC$. (For the `if' part, one can use a result of Baer
\cite[Theorem(1.9)]{FM}.) In later sections, we will therefore think of
$\widetilde \Gamma(\Si)$ as the set of pairs $(f,n)\in
\Gamma({  \Si})\times \BZ$ with multiplication given by
(\ref{comp}). But 
for now, 
it will be convenient
to think of elements of $\widetilde \Gamma(\Si)$ as extended mapping cylinders.
}\end{rem}

\begin{rem}\label{uni}{\em   The multiplication in (\ref{comp}) depends on
   $\lambda(\Si)$. Nevertheless, if $\Si$ and $\Si'$
   have the same underlying surface $\underline
   \Si=\underline\Si'$, then   $\widetilde\Gamma(\Si)$ and $\widetilde
   \Gamma({\Si'})$ are canonically isomorphic. The isomorphism is given by conjugating by $\I \times \underline \Si$ with
identity boundary identifications, but with the source and target
being respectively $\Si$ and $\Si'$.
}\end{rem}

\section{A surgery formula}\label{newsec3}

 Recall that the 
mapping class group  $\Gamma(  \Si)$ is generated by Dehn
twists.   If $ \alpha$ is 
an unoriented
simple closed curve in $\Si$, let 
$D({\alpha})$ denote the Dehn twist along $\alpha$. Our Dehn twists
are defined as in Birman \cite{Bi} ({\em i.e.,} they `turn right');
this is the opposite convention from the one in \cite{FM}.  Let
$\alpha_-$ denote the framed knot in $\I \times \Si$ given by $\frac 1
2 \times \alpha $  with framing $-1$ with respect to the `surface
framing' that this knot has  as a subset of the surface $ \frac 1 2
\times \Sigma$.

\begin{lem}\label{1twist} 
Let $\Si$ be an extended surface with lagrangian
$\lambda=\lambda(\Si)$. Let $ \alpha$ be a simple closed curve in $\Si$. Let $C({\alpha})\in \widetilde\Gamma(\Si)$ be
the result of extended surgery along the framed knot $\alpha_-$ on the identity
cobordism $\I \times \Si$ 
(with
 weight $w(\I \times \Si)=0$, and both ends 
equipped with 
$\lambda(\Si)$.) Then  

(i) the underlying cobordism is the mapping cylinder of the Dehn 
twist $D(\alpha)$.

(ii) Moreover, the weight of $C({\alpha})$ is given by 
\begin{equation} \label{wf}
w( C({\alpha})) =
\begin{cases}
-1 & \text{if \ }[\alpha] \in \lambda(\Si)\\
0 & \text{if \ } [\alpha] \notin \lambda(\Si)
\end{cases} 
\end{equation}  \end{lem}

 Here, $[\alpha] \in H_1(\Si;\BQ)$ denotes the homology class of
 $\alpha$ with an arbitrary orientation. Note that in the formulae
 above, replacing  $[\alpha]$ by $-[\alpha]$ has no effect.
\begin{proof} Statement (i) of the lemma is well-known, see {\em
    e.g. \cite{MR}}. Statement (ii) can be deduced from our more general
  surgery formula in Theorem~\ref{new4.1} below (see Remark~\ref{remcol}), but it can also be seen directly by the following
  Maslov index computation which was suggested to us by the referee.  Let
  $\alpha_0$ be the framed knot in $\I \times \Si$ given by $\frac 1
2 \times \alpha $  with the `surface
framing'. Let $E$ denote the exterior of a regular neighborhood
$\nu(\alpha_0)$  of
$\alpha_0$ in $\I\times\Si$. Its boundary $\partial E$ is the disjoint
union of $\partial \I\times
\Si$ and the torus $T=\partial(\nu(\alpha_0))$. The meridian of $\alpha_0$ and the
preferred parallel 
(=longitude) 
 of $\alpha_0$ defined by its framing are denoted by
$m(\alpha_0)\subset T$ and $p(\alpha_0)\subset T$
respectively.
We choose our meridian and 
 preferred parallel
so that $m(\alpha_0)\cdot p(\alpha_0)=1$, if $T$ is oriented
as the boundary of $\nu(\alpha_0).$  
 Then $$C(\alpha)=E\cup_f(D^2\times S^1)$$
where $f:S^1\times S^1\rightarrow \partial E$ is an
orientation-reversing homeomorphism sending $S^1 \times pt$ to
$p(\alpha_0) -m(\alpha_0)$ in homology. By
definition of extended surgery, we have 
\begin{equation} 
\notag
w(C(\alpha)) = 
0+0-\mu_{T}(L,\langle m(\alpha_0)\rangle, 
\langle p(\alpha_0) -m(\alpha_0)\rangle)\\
\end{equation}
where $\mu_{T}$ is Maslov index  and $L$ is the lagrangian in $H_1(T; \BQ)$ given by
those elements of $H_1(T; \BQ)$ which are homologous in $E$ to
some element of $0\times \lambda + 1\times \lambda$, where
$\lambda=\lambda(\Si)$. 
 If $\alpha$ belongs to $\lambda$, then $L=\langle p(\alpha_0)\rangle$
(since $p(\alpha_0)$ can be isotoped in $E$ to $1\times \alpha)$; a
simple computation straight from the definition of Maslov index gives 
\begin{equation}
\notag
\mu_{T}(L,\langle m(\alpha_0)\rangle,\langle p(\alpha_0) -m(\alpha_0)\rangle
) =
\Signature
[ (-m(\alpha_0)) \cdot p(\alpha_0)]= 1
\end{equation} in this first case. 
For this computation, $T$ is oriented as part of the boundary of $E$, and thus 
$m(\alpha_0) \cdot p(\alpha_0)=-1$.
 If, on the other hand, $\alpha$
does not belong to $\lambda$, we claim that $L=\langle
m(\alpha_0)\rangle$; assuming this for the moment, it follows that $$\mu_{T}(L,\langle m(\alpha_0)\rangle,\langle p(\alpha_0) -m(\alpha_0)\rangle
) = 0$$ in this second case (since two of
the three lagrangians are now the same, see Lemma~\ref{2lag}).  

To see that $L=\langle
m(\alpha_0)\rangle$ if $\alpha$
does not belong to $\lambda$, choose  
$x\in \lambda\cap H_1(\Si; \BZ)$
so that
$x\cdot \alpha\neq 0$
and $x$ is primitive.
We have that $x$
is represented by a simple closed curve $\gamma \subset \Si$, which we
may assume transverse to $\alpha$. Then $\I\times \gamma$ meets $\frac 1
2 \times \alpha $ non-trivially; cutting out from $\I\times \gamma$
small disks around the intersection points provides a surface
realizing a homology from a non-zero multiple of $m(\alpha_0)$ to some
element of $0\times \lambda + 1\times \lambda$. This shows that
$m(\alpha_0)$ lies in $L$, as asserted.
\end{proof}

Consider a word $\mathfrak
w=\prod_{i=1}^n {\alpha_i}^{\varepsilon_i}$, where $\varepsilon_i=
\pm 1$, and the $\alpha_i$ are unoriented simple closed curves in
$\Si$.
Let $D({\mathfrak w} )= \prod_{i=1}^n D(\alpha_i)^{\varepsilon_i} \in \Gamma
( {\Si})$. (Here $D(\alpha_1\alpha_2)=D(\alpha_1)\circ D(\alpha_2)$
means first apply $D(\alpha_2)$ then $D(\alpha_1)$.) Since Dehn twists generate
$\Gamma( {\Si})$, every mapping class $f$ is of the form
$D({\mathfrak w} )$ for some word $\mathfrak
w$. We now give a surgery formula for the product $$ C(\mathfrak w)= \prod_{i=1}^n
C({\alpha_i})^{\varepsilon_i}$$ in the extended mapping class group.
 Here, the product structure is composition of mapping cylinders as
defined in (\ref{comp}).

\begin{thm}\label{new4.1} If $f=D(\mathfrak w)$, then 
\begin{equation}
\notag
C(\mathfrak w)= C(f, n^0_\lambda(\mathfrak
w))~,
\end{equation} 
where 
 $n^0_\lambda(\mathfrak w)=\sigma(L^0_\lambda(\mathfrak
w)),$ the signature  
of the linking matrix of the framed link  $L^0_\lambda(\mathfrak
w)$ in $S^3$ which is constructed below. 
\end{thm}

The framed link  $L^0_\lambda(\mathfrak
w)$ is not uniquely determined by the word $\mathfrak w$ and the
lagrangian $\lambda$, but 
 the signature of its linking matrix
is. We construct  $L^0_\lambda(\mathfrak
w)$ in three steps. First, we  embed $\Si$ in $S^3$ so that it is
the boundary of a handlebody $\BH$ in $S^3$ such that $\lambda(\Si)=
\text{kernel }(H_1(\Si;\BQ) \rightarrow H_1(\BH;\BQ))$ and such that
the complement $S^3 \setminus
{\Int}(\BH)$ is another handlebody $\BH'$. If these conditions are
satisfied, we say that  $\Si$ is 
{\em  well placed} 
in $S^3$
with respect to $\lambda$. 
 
The second step is to decompose  $S^3 =\BH
  \cup (\I \times \Si) \cup \BH'$ where $\I \times \Si$ is a collar on
  the boundary, and to construct a framed link  $L(\mathfrak w)$ lying in $\I \times \Si \subset
S^3$. This is done,  as in  \cite[2.7]{MR}, by layering 
$-\varepsilon_i$-framed (with respect to the surface framing)
copies of $\alpha_i$, 
starting with 
$\alpha_n$
near $\{0\} \times \Si$, 
then $\alpha_{n-1}$ 
 and so on, moving outward until 
$\alpha_{1}$ is inserted near $\{1\}
\times \Si$.~\footnote{The reason for inserting the $\alpha_i$ in this
  order is that the composition of mapping cylinders
 {\em first $\underline C(f)$, then $\underline C(g)$}
 is $\underline C(g\circ f)$.} 
(Here, the orientation of the individual link
components is chosen arbitrarily. It will not play a 
role 
in what
follows.)

Finally, for the third step, let  $g$ denote the genus of $\Si$.
Choose simple closed oriented curves $m_1,\ldots m_g$, $\ell_1,\ldots
\ell_g$ such that each $m_i \cap \ell_i$  consists of  
one transverse intersection point
(and $m_i
\cdot \ell_i=1$ for the given orientation of $\Si$) but the $m_i$ and $\ell_j$ are otherwise disjoint.
Moreover the 
$m_i$ should bound disjoint disks in $\BH$, and 
the $\ell_j$ should bound disjoint disks in $\BH'= S^3 \setminus \Int
\BH.$ We refer to the $m_i$ as the meridians of $\BH$.  See
Figure~\ref{ZZZ}.

\begin{figure}[h]
\includegraphics[height=1.2in]{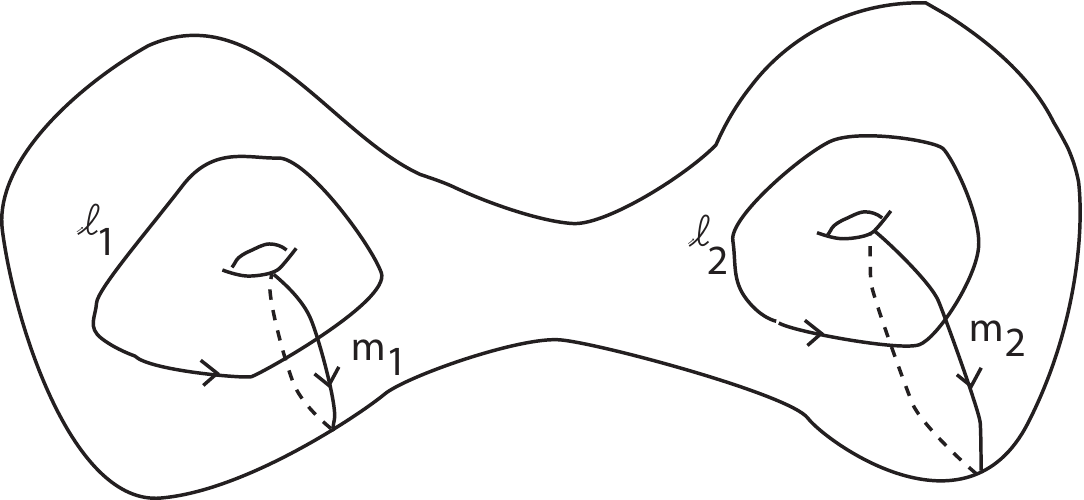}
\caption{ $m_1$, $\ell_1$, $m_2$, $\ell_2$ on  $\Si$ of genus two. The lagrangian $\lambda(\Si)$ is spanned by $m_1$ and $m_2$.}
\label{ZZZ}
\end{figure}

Consider the zero-framed  unlink $U$ with $g$ components obtained by
pushing the meridians  $m_1,\ldots m_g$ of $\BH$ up into  $\BH'$ in
$S^3$. 

\begin{de} We let $L^0_\lambda(\mathfrak w)$ be the $(n+g)$-component
  link in $S^3$ whose first
$n$ components are $L(\mathfrak w)$  sitting in $\I \times \Si \subset
S^3$, and whose  later components are given by the zero-framed
unlink $U$ sitting in $\BH'$.

\end{de}

\begin{proof}[Proof of Theorem~\ref{new4.1}]
 Observe that $ C(\mathfrak w)$ is the result
  of extended surgery on  $\I \times \Si$ along $L(\mathfrak w)$.
This follows from the associativity of extended gluing.
We need to show that the
  weight of $C(\mathfrak w)$ is equal to the signature of 
the linking matrix of 
  $L^0_\lambda(\mathfrak w)$:
\begin{equation}\label{weightC}
w(C(\mathfrak w)) = \sigma(L^0_\lambda(\mathfrak w))~.
\end{equation}

Consider the decomposition $S^3 =\BH
  \cup (\I \times \Si) \cup \BH'$. Make $\BH$ and $\BH'$ into extended
  manifolds by giving them weight zero. 
Let $Y$ be the result of extended gluing $\BH \cup C(\mathfrak w) \cup \overline
\BH $, 
where the source surface of the extended mapping cylinder $C(\mathfrak w)$ is glued to the boundary of
$\BH$, and the target surface of  $C(\mathfrak w)$ is glued to the boundary
of $\overline \BH $. 
 Since $w(\BH)=0$, 
we also have $w(\overline\BH)=0$, and hence 
\begin{align} \notag
w(Y)&= w(\BH) + w(C(\mathfrak w)) +w(\overline \BH) + \mu(\la,\la, D(\mathfrak w)_\star^{-1}(\la))
+ \mu(D(\mathfrak w)_\star(\la), \la, \la) \\
&=w(C(\mathfrak w))~.\notag
\end{align} Here, the two
Maslov index terms are zero, because in both cases two of the three
lagrangians coincide (see Lemma~\ref{2lag}). 

Let $(\BH')_U$ denote the result of extended surgery
on $\BH'$ along the zero-framed unlink $U$.
Then extended gluing $\BH \cup
(\BH')_U$ gives $(S^3)_U$, which is ${\#}^g S^1 \times
S^2$ (the connected sum of $g$ copies of $S^1 \times
S^2$)  with weight zero (use Lemma~\ref{4intrep} for the weight
computation). On the other hand, 
 extended gluing
$\BH \cup \overline \BH$ is also ${\#}^g S^1 \times S^2$ with weight zero, as follows from a Maslov index computation 
like 
the one 
for $w(Y)$
given above. This shows that the 
standard identification of $(\BH')_U$ with $\overline \BH$ holds true as
extended manifolds. Thus 
\begin{align}\notag w( Y)&=w(\BH \cup C(\mathfrak w) \cup \overline
\BH) =  w(\BH \cup C(\mathfrak w) \cup 
(\BH')_U) = w((S^3)_{L^0_\lambda(\mathfrak w)} )\\ 
&=\sigma(L^0_\lambda(\mathfrak w)) \notag 
\end{align}  
where we have again used Lemma~\ref{4intrep} in the last equality.
This proves the equality (\ref{weightC}), since both of its sides are
equal to $w(Y)$.
\end{proof}

\begin{rem}\label{remcol} {\em 
If the word $\mathfrak w$ has length $n=1$, say $\mathfrak w = \alpha$,
  then $L(\mathfrak w)$ is the framed knot $\alpha_-$, and the
  signature of 
the linking matrix of 
  ${L^0_\lambda(\mathfrak w)=L^0_\lambda(\alpha)}$ is easily computed,
  as follows. Suppose the homology class $[\alpha]= \sum_{i=1}^{g} (a_i [m_i] + b_i
[\ell_i])$, with integers $a_i$ and $b_i$ ($i=1,\ldots, g$). (Here, we have picked
an arbitrary orientation of the curve $\alpha$.) Let
$\alpha'$ be a parallel copy of $\alpha$ on one of the layers  $\{t\}
\times \Si$ (for $t\not= \frac 1 2$). Then the linking number $\Lk(\alpha, \alpha')=  \sum_i
a_i b_i.$ Thus the framing of the first component of ${L^0_\lambda(\alpha)}$ is  $-1 +   \sum_i a_i b_i.$
The linking number of the first component of ${L^0_\lambda(\alpha)}$ with the
$(i+1)$th component is $b_i$. The lower right $g \times g$
block of the linking matrix of ${L^0_\lambda(\alpha)}$ consists of zeros. Note
that by construction, the lagrangian $\lambda$ is the span of the
meridians 
$m_i$. If $[\alpha] \in \lambda$, then all the $b_i$'s are zero, and
$\sigma({L^0_\lambda(\alpha)})=-1.$ If $[\alpha] \notin \lambda$, then some
$b_i \ne 0$, and $\sigma({L^0_\lambda(\alpha)})=0.$ This computation
together with  Theorem~\ref{new4.1} provide another proof of Formula
(\ref{wf})  for the weight $w(C(\alpha))$ in Lemma~\ref{1twist}.

See
Figures~\ref{ZZZ1} and~\ref{ZZZ2} for a concrete example. 
}\end{rem} 

 \begin{figure}[h]
\includegraphics[height=1.2in]{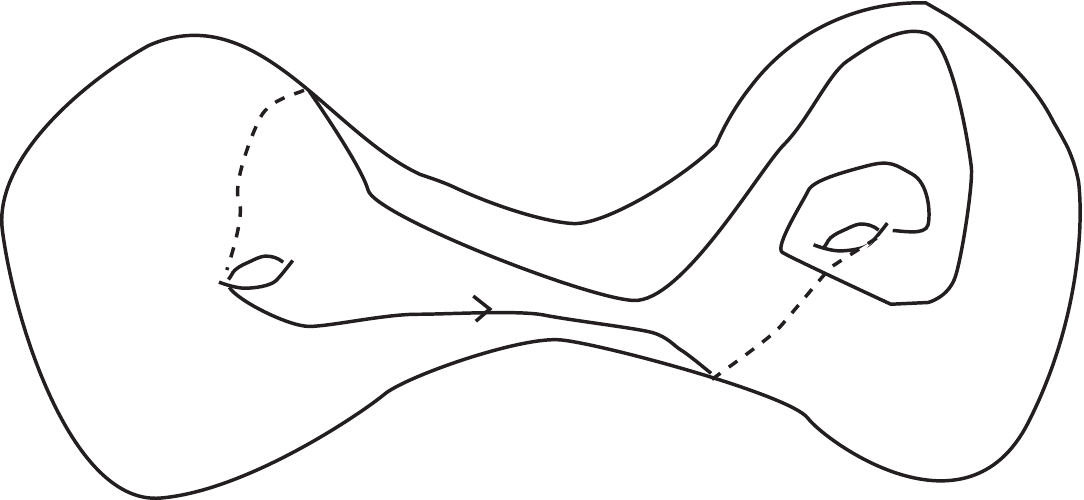}
\caption{A curve $\alpha$ with 
      $[\alpha]=  m_1+\ell_1+m_2 +2\ell_2$.
One has $\Lk(\alpha, \alpha')=3$, so that the framing specified by the surface is the
  `$3$-framing' in this case.}
\label{ZZZ1}
\end{figure} 

 \begin{figure}[h]
\includegraphics[height=1.2in]{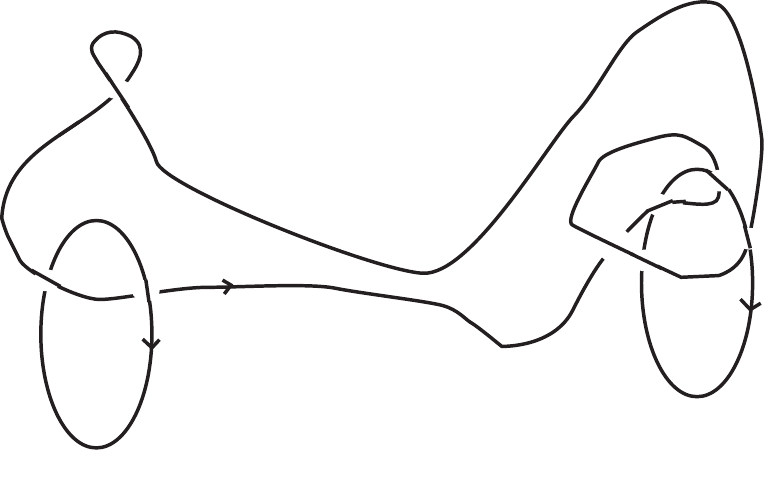}
\caption{The framed link $L^0_\lambda(\alpha)$ indicated with the `blackboard
  framing' convention. The framing of $\alpha_-$ is $2$.
  One has $\sigma({L^0_\lambda(\alpha)})=0$.}
\label{ZZZ2}
\end{figure}

\section{Definition of the subgroup $\widetilde\Gamma(\Si)^{++}$}\label{newsec5}
Recall that for every Dehn twist $D(\alpha)$, we defined a
preferred lift $C(\alpha)$ to $\widetilde\Gamma(\Si)$. Shifting the
weight by one, we define $$W({\alpha})= W \circ C({\alpha})$$ (recall $W=C(\Id_\Si,1)$). 
By Lemma~\ref{1twist}, we have
\begin{equation}\label{41} w( W({\alpha})) =
\begin{cases}
0 & \text{if \ } [\alpha] \in \lambda(\Si)\\
1 & \text{if \ } [\alpha] \notin \lambda(\Si)
\end{cases} 
\end{equation}
Note that if the curve $\alpha$ bounds a disk, then $C(\alpha)=W^{-1}$
but $W(\alpha)=1$ is the identity element of
$\widetilde\Gamma(\Si)$. 
\begin{de}\label{ind4} The group $\widetilde\Gamma(\Si)^{++}$ is defined to be the
  subgroup of  $\widetilde\Gamma(\Si)$ 
   generated 
   by the lifts $W({\alpha})$ 
  for all (isotopy classes of) simple closed curves $\alpha$ on
  $ {\Si}$. 
\end{de}   

The reason for the superscript $++$ in the notation
$\widetilde\Gamma(\Si)^{++}$ will become clear later (see
Remark~\ref{new76}). 

Given a word  $\mathfrak
w=\prod_{i=1}^n {\alpha_i}^{\varepsilon_i}$, we denote its exponent
sum by $e(\mathfrak w)=
\sum_{i=1}^n {\varepsilon_i}$, and we write $W(\mathfrak
w)= \prod_{i=1}^n
W({\alpha_i})^{\varepsilon_i}$. We have the following immediate
corollary of Theorem~\ref{new4.1}.

\begin{cor} \label{new52}If $f=D(\mathfrak w)$, then
\begin{equation} 
\notag
W(\mathfrak w)= C(f, n_\lambda(\mathfrak
w))~,
\end{equation} 
where $n_\lambda(\mathfrak w)= e(\mathfrak w) + n^0_\lambda(\mathfrak
w)= e(\mathfrak w) +\sigma(L^0_\lambda(\mathfrak w))$.
\end{cor}

In the rest of this section, we show that $W^4\in
  \widetilde\Gamma(\Si)^{++}$. Thus, $\widetilde\Gamma(\Si)^{++}$  has
  index at most four in $\widetilde \Gamma({ 
    \Si})$. In the next section, we will see that the index of
  $\widetilde\Gamma(\Si)^{++}$  in $\widetilde\Gamma(\Si)^{}$  is
  equal to four.

\begin{lem}\label{relator} If $\mathfrak w$ is a relator in the mapping
  class group (i.e., if $D({\mathfrak w} )=\Id_\Si$), then 
  $\sigma(L^0_\lambda(\mathfrak w) )=\sigma(L_\lambda(\mathfrak w))$,
  where $L_\lambda(\mathfrak w)$ is obtained from $L^0_\lambda(\mathfrak w)$ by omitting the zero-framed
unlink $U$.
\end{lem}

\begin{proof}  As in the proof of Theorem~\ref{new4.1}, recall that $\BH \cup (\I \times \Si) \cup \BH'$ is
  $S^3$ with weight zero, since we assume $w(\BH)=w(\BH')=0$. Now consider the extended gluing
  $X=\BH \cup  C({\mathfrak w}) \cup \BH'$.   Since $D({\mathfrak w}
  )=\Id_\Si$, we have that $X$ is $S^3$ with some weight. We compute this
  weight in two ways. On the one hand, $X$ is extended surgery on
  $S^3$ along the link $L_\lambda(\mathfrak w)$, hence 
\begin{equation}\label{W1}
w(X)=\sigma(L_\lambda(\mathfrak w))
\end{equation} by Lemma~\ref{4intrep}. On the other hand,  the fact
  that $D({\mathfrak w}
  )=\Id_\Si$ implies that 
we have strict additivity when computing the weight of the gluing in $X$
(the two
Maslov index terms are zero, because for both of them two of the three
lagrangians coincide). Since moreover $w(\BH)=w(\BH')=0$, we have
\begin{equation}\label{W2}
w(X)
= w(C({\mathfrak
    w})) = \sigma(L^0_\lambda(\mathfrak w))~,
\end{equation} where we used (\ref{weightC}) in the last equality.  Comparing (\ref{W1}) and (\ref{W2}), the lemma
follows.
\end{proof}

\begin{lem}\label{55} We have $W^4\in \widetilde\Gamma(\Si)^{++}$.
\end{lem}
\begin{proof} There exists a relator $\mathfrak u$ with 
  $e(\mathfrak u)=11$
and
  $\sigma(L_\lambda(\mathfrak u))=-7$.  This relator lives on a
  one-holed torus (embedded into $\Si$ in an arbitrary fashion) and is described in more detail in the
  proof of Proposition~\ref{8.7}. Using Lemma~\ref{relator}, we have
  $n_\la(\mathfrak u)=4$ and therefore $W(\mathfrak
  u)=C(\Id_\Si, 4)=W^4$ by Corollary~\ref{new52}. Thus $W^4\in
  \widetilde\Gamma(\Si)^{++}$. 
\end{proof}

\section{Algebraic description of $\widetilde\Gamma(\Si)^{++}$}
\label{sec5}

In this section, we give a purely algebraic description of 
$\widetilde\Gamma(\Si)^{++}$.
When it should cause no confusion, we
use the same letter to denote a mapping class group element and its
induced map on the rational first homology of the surface. Unless
otherwise stated, all homology groups are with rational coefficients.  
As before, we write $\cdot$ for the intersection form on $H_1(\Si)$.   
We denote
the lagrangian $\lambda(\Si)$ simply by $\lambda$.

Our algebraic description of 
$\widetilde\Gamma(\Si)^{++}$ 
uses the bilinear form $\star_f$ described in the following
lemma. This form was introduced by Turaev \cite{T2,T3}.
 
\begin{lem} 
\label{starfdef} 
(Turaev \cite[2.1,2.2]{T3})
If $f \in \Gamma({  \Si})$, then
\[a \star_{f} b =(f-1)^{-1}(a) \cdot b  \] is a well-defined non-singular  bilinear form on $(f-1)H_1({\Si}).$
\end{lem}

Here, $(f-1)^{-1}(a) \cdot b$ means $x\cdot b$ where $x$ is any
element of $(f-1)^{-1}(a) $.

\begin{proof} Suppose that $x_1,x_2 \in (f-1)^{-1}(a)$. We  let
  $x=x_1-x_2.$ To see 
 that   
$\star_{f}$  is well-defined on  $(f-1)
  H_1({\Si})$,  we need to see that 
\begin{equation}\label{sttq} x \cdot b=0
\end{equation}
  provided $b=(f-1)(y)$ for some $y\in H_1({\Si})$. This
  is shown as follows. Since $(f-1)(x)=0$, we have  $f(x)=x$, and hence
  also $f^{-1}(x)=x$. Using that the intersection form $\cdot$ is
  preserved by $f^{-1}$,
  the following computation proves  (\ref{sttq}): 
$$ x \cdot b =x \cdot f(y) - x \cdot y=f^{-1}(x) \cdot y - x \cdot y= x \cdot y - x \cdot y=0~.$$

 To show non-singularity of the form $\star_f$, observe that 
$$ (f-1)(a) \cdot b= f(a) \cdot b - a \cdot b = f(a) \cdot b -
f(a)\cdot f(b) = -f(a) \cdot
 (f-1)(b)$$ for all
 $a,b \in H_1({\Si})$. Hence the kernel of $f-1$ is contained in the annihilator (with respect to
 $\cdot$) of $(f-1)H_1(\Si)$. Counting dimensions, it follows that the
 kernel of $f-1$ is equal to this annihilator. This proves that
 $\star_f$ is non-singular on $(f-1)
  H_1({\Si})$.

\end{proof}

\begin{de}(Turaev) We define  
 $\sgn[\det(\star_{f})]$ 
to be the sign of the
  determinant of a matrix for $\star_f$ with respect to a basis of $(f-1)H_1(\Si)$.
\end{de}
 
Note that 
 $\sgn[\det(\star_{f})]$ 
does not depend on the choice of
the basis. The 
form  $\star_{f}$ is neither symmetric 
 nor
 skew-symmetric in general,
but the definition of 
$\sgn[\det(\star_{f})]$ makes sense. Since
$\star_{f}$ is non-singular, we have 
 $$\sgn[\det(\star_{f})]=\pm 1~.$$
Here, let us agree that 
$\sgn[\det(\star_{\Id})]=1$ ({\em i.e.,} the
determinant of a $0 \times 0$ matrix should be taken to be one.)

\begin{rem}{\em 
Turaev \cite{T2,T3} denotes 
 $\sgn[\det(\star_{f})]$ by
    $\varepsilon(f)$.
}\end{rem}

We need the following simple observation.
 
\begin{lem} For every lagrangian $\la\subset H_1(\Si)$, the
  restriction of the form $\star_{f}$ to $\lambda
     \cap (f-1)H_1(\Si)$ is symmetric.
\end{lem}
\begin{proof} Suppose that $(f-1)x=a \in \lambda$, and  $(f-1)y=b\in
\lambda$, then
\[ a\star_{f}b-b\star_{f} a=  x\cdot (f(y)-y) -y \cdot (f(x)- x)= x \cdot
f(y) +  f(x)  \cdot y-2 x \cdot y \]
On the other hand,
\[0=b \cdot a= (f(y) -y)\cdot (f(x)-x)= x \cdot f(y) +  f(x)  \cdot y-2 x
\cdot y \] Thus $a\star_{f}b=b\star_{f} a$, as asserted.
\end{proof}

\begin{de} Let $\star_{f,\lambda}$ denote the restriction of the form $\star_{f}$ to $\lambda
     \cap (f-1)H_1(\Si)$. We denote the signature of this form by 
 $\Signature(\star_{f,\lambda})$.
\end{de}

We can now state the main result of this section. 

   \begin{thm}\label{hom} Given $f\in\Gamma(\Si)$ and a lagrangian
     $\lambda$, define 
\begin{equation}  
\notag
n_\la(f) =
\Signature(\star_{f,\lambda})
-\dim((f-1)H_1(\Si))-
\sgn[\det(\star_{f})]  +1~.
\end{equation}
Then the set $$\{ C(f,n)\, | \, f\in\Gamma(\Si), \ n\equiv
n_\la(f) \pmod 4 \}$$
is an index four subgroup of $\widetilde\Gamma(\Si)$. 
 \end{thm}

The proof of Theorem~\ref{hom} will be given in 
Section~\ref{sec5b}.
The following  
corollary
relates the algebraic approach of
the present section with the 
approach {\em via} extended surgery
of
the previous sections.

\begin{cor}\label{hom-ind} The index four subgroup of
  $\widetilde\Gamma(\Si)$ given 
  in Theorem~\ref{hom} is equal to the subgroup  $\widetilde\Gamma(\Si)^{++}$
defined in Section~\ref{newsec5}.
\end{cor}

\begin{proof} Let $S$ denote the index four subgroup of
  $\widetilde\Gamma(\Si)$ given 
  in Theorem~\ref{hom}. Recall that $\widetilde\Gamma(\Si)^{++}$ was defined as the
subgroup of $\widetilde\Gamma(\Si)$ generated by the elements
$W(\alpha)$ which were certain 
lifts of the Dehn twists $D(\alpha)$ to $\widetilde\Gamma(\Si)$. It is
easy to check 
from (\ref{41}) that if $n$ is the weight of $W(\alpha)$ then $n\equiv
n_\la(D(\alpha)) \pmod 4$. Thus
\begin{equation}\label{indS}
\widetilde\Gamma(\Si)^{++} \subset S~.
\end{equation} But since $S$ has index four in
$\widetilde\Gamma(\Si)$, and we know that $W^4=C(\Id_\Si,4) \in
\widetilde\Gamma(\Si)^{++}$ by Lemma~\ref{55}, the inclusion
(\ref{indS}) is an equality.
\end{proof}

The following corollary 
is immediate. 
\begin{cor} If $f$ is given as a word in Dehn twists
  $f=D(\mathfrak w)$, then the integer $n_\lambda(\mathfrak w)=
  e(\mathfrak w) +\sigma(L^0_\lambda(\mathfrak w))$ defined in
  Corollary~\ref{new52} satisfies the congruence
$$n_\lambda(\mathfrak w)\equiv n_\la(f) \pmod 4~.$$
\end{cor}
Thus, the group $\widetilde\Gamma(\Si)^{++}$ can also be described as
the subgroup of $\widetilde\Gamma(\Si)$ consisting of the
  $C(f, n)$ where $n\equiv n_\la(\mathfrak w)\pmod 4$, where
  $\mathfrak w$ is any word representing the mapping class $f\in
  \Gamma(\Si)$.

\begin{rem}{\em We briefly sketch a second way to see that
    $\widetilde\Gamma(\Si)^{++}$ is an index four subgroup of
    $\widetilde\Gamma(\Si)$. This second proof does not use
    Theorem~\ref{hom}, but uses a presentation of the mapping class
    group. We start again with the fact, shown in  Lemma~\ref{55}, that $W^4=C(\Id_\Si,4) \in
\widetilde\Gamma(\Si)^{++}$. It remains to show that if
  $W^n$ is in the kernel of the forgetful map
$\widetilde\Gamma(\Si)^{++}\rightarrow  \Gamma({  \Si})$, then
$n\equiv 0 \pmod 4$. Using Corollary~\ref{new52} and Lemma~\ref{relator}, we see that we need to show that
    $$e(\mathfrak w) +\sigma(L_\lambda(\mathfrak w))\equiv 0 \pmod 4$$
for every relator $\mathfrak w$ in a presentation of $\Gamma( {\Si})$  with
all Dehn twist as generators. This computation was done in a somewhat
different context in \cite[Proposition 3.4
(ii)]{MR}, the main difference being that there were no lagrangians in
\cite{MR}. But if $\mathfrak w$ is a
relator, then our  $\sigma(L_\lambda(\mathfrak w))$
is equal to the number $\sigma_b(\mathfrak w)$ defined in
\cite{MR}. The key to seeing this is to observe that 
if $\mathfrak w$ 
    is a relator, then $\sigma(L_\la(\mathfrak w))$ does not depend on the lagrangian
    $\la$. This can be seen as follows. If $\mathfrak w$
    is a relator, then $\sigma(L_\la(\mathfrak
    w))=\sigma(L_\la^0(\mathfrak w))$ is the weight of $C(\mathfrak w)$
    by equation 
    (\ref{weightC}) (proved in the proof of Theorem~\ref{new4.1}).
    Changing the lagrangian amounts to conjugating
    in the way explained in
    Remark~\ref{uni}. Since $C(\mathfrak w)$ has underlying manifold $\I\times
    \Si$, its weight is not changed by conjugating. 
}\end{rem}

\section{The index two subgroup $\widetilde\Gamma(\Si)^{+}$ of
  $\widetilde\Gamma(\Si)$}\label{newsec7}
The following corollary of Theorem~\ref{hom} is immediate.

\begin{cor}\label{2G} The set $$\{ C(f,n)\, | \, f\in\Gamma(\Si), \ n\equiv
n_\la(f) \pmod 2 \}$$
is an index two subgroup of $\widetilde\Gamma(\Si)$.
\end{cor}

In the remainder of this section, we show that this subgroup is equal
to the extension $\widetilde\Gamma(\Si)^{+}$ constructed by one of us in
\cite{G}. It was defined as follows.

  \begin{de} \label{ind2} ( \cite{G}) Let $\widetilde\Gamma(\Si)^+$ be the subset of
  $\widetilde\Gamma(\Si)$ given as
\begin{equation} 
 \widetilde\Gamma(\Si)^+= \{  C(f, n )
\, | \,
 f \in  \Gamma(
 \Si) , \ n\equiv \genus(\Si) +\dim \bigl(       \lambda      \cap  f(\lambda     )
\bigr)\pmod{2} .\}
 \notag
\end{equation}  
 \end{de} 
It was shown in \cite{G} that $\widetilde\Gamma(\Si)^+$ is a subgroup of
$\widetilde\Gamma(\Si)$.  Our work allows one to give a new proof of
this fact (see Remark~\ref{3G}).  Note that $\widetilde\Gamma(\Si)^+$ has index two in
$\widetilde\Gamma(\Si)$.

 \begin{prop}
\label{mod2} For every $f$ and $\la$, we have
\begin{equation}
\notag
      \Signature(\star_{f,\lambda})
+ \dim((f-1)H_1(\Si)) \equiv 
\dim \la
       +\dim(\la \cap f      (\la))
\pmod{2}~.
\end{equation}
\end{prop}

      \begin{proof}
Let us write $V=(f-1)H_1(\Si)$ and $E=\ker(f-1)$. As
    shown in the proof of Lemma~\ref{starfdef}, we have $V=E^\perp$
    with respect to the form $\cdot$ on $H_1(\Si)$. The domain of
    definition of the form $\star_{f,\lambda}$ is $\la \cap V$.  
The radical of the form $\star_{f,\lambda}$ is given
    by 
\begin{equation}\label{s1}
\rad(\star_{f,\lambda})= \la \cap (f-1)(\la)~.
\end{equation} 
To see this, notice that an
    element  $a=(f-1)(x)\in
    \la\cap V$ is in the radical of $\star_{f,\lambda}$ 
if and only if
$x\cdot b=0$ for all $b\in
    \la\cap V$. Thus $$\rad(\star_{f,\lambda})= \la \cap (f-1)\bigl(
    (\la \cap V)^\perp \bigr)~.$$ But $(\la \cap V)^\perp= \la^\perp +
    V^\perp= \la + E$ and $(f-1)(E)=0$. This proves (\ref{s1}).
 Now observe that 
\begin{equation} \label{s2}
\Signature(\star_{f,\lambda}) 
\equiv \rank (\star_{f,\lambda}) 
 =  \dim
(\la\cap V) - \dim \rad(\star_{f,\lambda}) \pmod 2~.
\end{equation} Using (\ref{s1}), the exact sequence
$$ \begin{CD}
 0 @>>> \la \cap E  @>>> \la \cap f(\la) @>{1-f^{-1}}>>\la
 \cap (f-1)(\la)@>>> 0 
\end{CD}$$ shows that 
\begin{equation}\label{s3}
\dim \rad(\star_{f,\lambda})= \dim (\la \cap
f(\la)) - \dim (\la \cap E)~.
\end{equation} But $\la \cap E= \la^\perp \cap
V^\perp=(\la+V)^\perp$, hence 
\begin{equation}\label{s4}
\dim(\la \cap E) \equiv \dim (\la+V)
\pmod 2~.
\end{equation} Putting together (\ref{s2}), (\ref{s3}), and
(\ref{s4}), we have 
$$ 
\Signature(\star_{f,\lambda})  
\equiv \dim
(\la\cap V) + \dim (\la \cap
f(\la)) + \dim (\la+V) \pmod 2~.$$ This implies  
Proposition~\ref{mod2}  
because of the equality 
\begin{equation*} 
\dim
(\la\cap V)+\dim (\la+V)= \dim V + \dim \la~. 
\qedhere
\end{equation*}
\end{proof}

\begin{cor}\label{alt} We have
     \begin{align}\notag
 \widetilde\Gamma(\Si)^+&= \{  C(f, n ) \, | \,  f \in  \Gamma(\Si) , \
 n\equiv 
\Signature(\star_{f,\lambda})  
+ \dim((f-1)H_1(\Si))\pmod{2}~\}\\
\notag
& =\{ C(f,n)\, | \, f\in\Gamma(\Si), \ n\equiv
n_\la(f) \pmod 2 \}
\end{align}
 \end{cor}

\begin{proof} This follows immediately from the 
      proposition, the fact that $\genus(\Si)=\dim \lambda$,  
and the
  definition of $n_\la(f)$.\end{proof}

\begin{rem}\label{3G}{\em Together with Corollary~\ref{2G}, this provides a new proof of the fact that the
    subset of $\widetilde\Gamma(\Si)$ defined in Definition~\ref{ind2} is a
    group. The original proof of this fact in \cite{G} was to
    construct an index two subcategory (called the {\em even subcategory} in
    \cite{G})
    of the extended cobordism category $\cC$ and to show that the
    intersection of  $\widetilde\Gamma(\Si)$ with this subcategory is precisely
    the subset given in Definition~\ref{ind2}. 
We don't know whether
    there exists an index four subcategory of $\cC$ whose
    intersection with $\widetilde\Gamma(\Si)$ is the group
    $\widetilde\Gamma(\Si)^{++}$.
}\end{rem}

\begin{rem}\label{new76}{\em Summarizing, we have the following  commutative diagram with exact
rows, where the vertical maps are all inclusions:
\[
 \begin{CD}
 0 @>>> 4\BZ @>>> \widetilde\Gamma(\Si)^{++} @>>>  \Gamma({  \Si})@>>> 1    \\
 @.          @VVV         @VVV                                     @|           
          @.   \\
0 @>>> 2\BZ @>>>  \widetilde\Gamma(\Si)^+  @>>>  \Gamma({  \Si})@>>> 1    \\
@.          @VVV         @VVV                                     @|            
         @.   \\
0 @>>> \BZ @>>> \widetilde \Gamma(\Si)  @>>>  \Gamma({  \Si})  @>>> 1 .   
\end{CD}
\] 

}\end{rem}

\section{Proof of Theorem~\ref{hom}.} \label{sec5b}
We begin the proof with some preliminary material. 
The set-theoretical section  $f \mapsto C(f,0)$ of $\widetilde \Gamma({\Si}) \rightarrow \Gamma({  \Si})$ can be used to define a  $2$-cocycle  on
$\Gamma({  \Si})$ whose cohomology class in $H^2(
\Gamma({  \Si}); \BZ)$ classifies the extension. The
extension and thus the cocycle depend  on $\la(\Si)$, which we are
denoting by $\la$. We will denote the cocycle by  $m_{\la}$ and call
it the {\em Maslov cocycle}.  In general, the $2$-cocycle defined by a section $s$ is given by $(g,f)\mapsto s(g)s(f)s(g\circ f)^{-1}$. 
By (\ref{comp}), 
we have that
\begin{equation}\label{ml1}
m_{\la}(g,f)= \mu_\Si\left( \la ,g (\la) , 
        (g \circ f)(\la) 
\right)=
-\mu_\Si\left( 
        ( g \circ f)(\la)
,g (\la) ,\la  \right)~.
\end{equation}

Next, we 
recall the well-known  4-manifold interpretation of the
Maslov index $\mu_\Si(\la_1,\la_2,\la_3)$ of three lagrangians. See
for example \cite[Section 12]{CLM}. Let $\mathcal{H}_{i}$ denote a
handlebody with boundary  $\Si$ such that $\la_i$ is the kernel of the
map $H_1(\Si) \rightarrow H_1( \mathcal{H}_{i})$  induced by
inclusion. Consider the 4-manifold  
$U(\mathcal{H}_{1},\mathcal{H}_{2},\mathcal{H}_{3})$ 
obtained by gluing   to $D^2 \times \Si$  (with the product
orientation using the standard orientation on $D^2$) three thickened
handlebodies 
$ \I \times \mathcal{H}_{i}$,
in the cyclic order indicated in Figure~\ref{u}. 

\begin{figure}[h]
\includegraphics[height=1.2in] {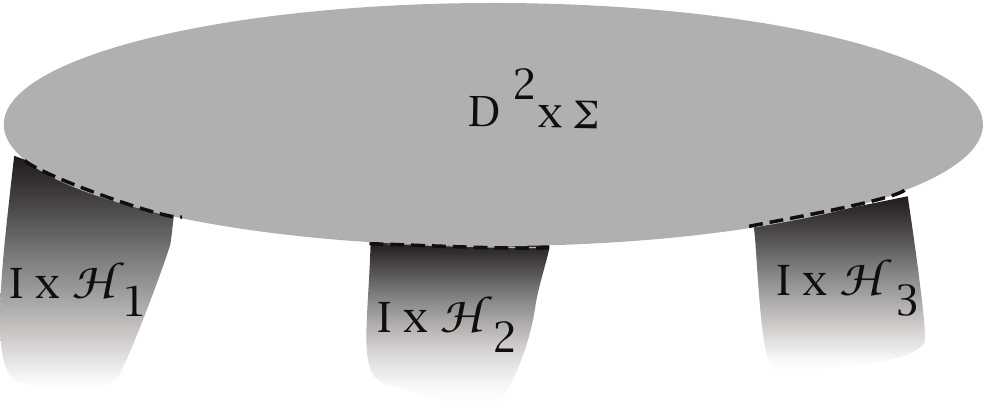}
\caption{A picture of 
$U(\mathcal{H}_{1},\mathcal{H}_{2},\mathcal{H}_{3}).$ 
The large oval disk represents $D^2 \times \Si$. The three dotted
lines represent copies of $\I \times \Si$ along which are glued the
thickened  handlebodies $ \I \times\mathcal{H}_i$. The boundary of 
$U(\mathcal{H}_{1},\mathcal{H}_{2},\mathcal{H}_{3})$ 
has three connected components. To indicate this, the thickened handlebodies    are drawn fading away. }
\label{u}
\end{figure}

By the Wall non-additivity formula, 
 the signature of the intersection 
 form
on the 4-manifold $U(\mathcal{H}_{1},\mathcal{H}_{2},\mathcal{H}_{3})$
 is given by 
the Maslov index of the three lagrangians:
\begin{equation}\label{ml2}
\Signature (
U(\mathcal{H}_{1},\mathcal{H}_{2},\mathcal{H}_{3}))
=\mu_\Si(\la_1,\la_2,\la_3)~.
\end{equation}
Note that in contrast with (\ref{weightunion}),  there is no
minus sign 
in (\ref{ml2}).

For the rest of this section,  $\BH$
will be a 
fixed
handlebody with  $\partial \BH= \Si$ such that the kernel of
$H_1(\Si) \rightarrow H_1(\BH)$ is the given lagrangian 
$\lambda=\lambda(\Si)$.
For $f\in\Gamma(\Si)$, we let $\BH_f$ denote the handlebody $\BH$ but
with boundary identified with $\Si$ through
$f^{-1}:\Si\rightarrow\Si=\partial \BH$; we therefore have  $$
\ker (H_1(\Si) \rightarrow H_1(\BH_f))=
f(\lambda)~.$$

\begin{prop} 
$m_{\la}(g,f) = 
 - \Signature (U(\BH_{g\circ f}, \BH_g, \BH) )~.$
\end{prop}

\begin{proof} This follows from (\ref{ml1}) and (\ref{ml2}). \end{proof}

 One can obtain  
$U(\BH_{g\circ f}, \BH_g, \BH)$ from $U(\BH, \BH, \BH)$
by cutting along two  arcs running across  the disk  and regluing using $f$ and $g$ as in Figure \ref{f}.
 In this figure (and in similar figures below)  the arrow labelled  $f$ indicates the direction of gluing: a point $x$ on the boundary component corresponding to the tail side of the arrow is  glued to  the point $f(x)$ on the boundary component corresponding to the head of the arrow.

\begin{figure}[h]
\includegraphics[height=1.2in]{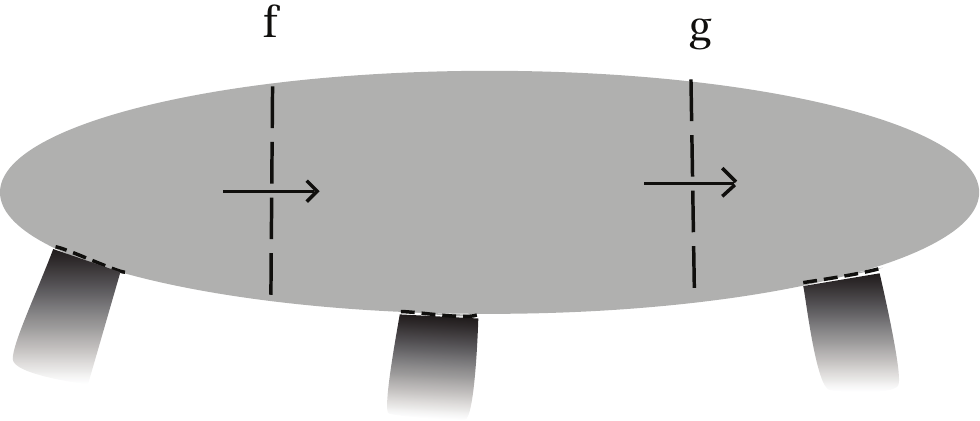}
\caption{ The three thickened handlebodies $\I \times \BH$  are
  attached first to form 
$U(\BH, \BH, \BH)$.
 Then  the
  manifold 
$U(\BH, \BH, \BH)$
is cut along the dotted seams
  and reglued  using $f$ and $g$. The result is oriented diffeomorphic
  to 
      $U(\BH_{g\circ f}, \BH_g, \BH)$        .}
\label{f}
\end{figure}

 The boundary of 
$U(\BH_{g\circ f}, \BH_g, \BH)$
consists of three
 {\em Heegaard manifolds} defined as follows. 
 The Heegaard manifold of $f \in \Gamma({  \Si})$ 
relative to the handlebody $\BH$, 
 denoted
 $\Heeg(f)$, is the quotient space  
$(\BH \sqcup \overline \BH)/\sim$  
where
 $\sim$ is the equivalence relation given by identifying  $x \in
 \partial \BH$ with    
$f(x) \in \partial \overline \BH$. 
Here, $\BH$ is
 oriented so that it induces the  given orientation on $\Si$. Using
 that  
 $\Heeg(f^{-1})=\overline{\Heeg(f)}$, 
one easily checks the 
following:
\begin{prop}   The   oriented boundary of 
 $U(\BH_{g\circ f}, \BH_g, \BH)$
is    
$$\overline{\Heeg(g \circ f)} \sqcup  {\Heeg(f)} \sqcup  {\Heeg(g)}~.$$ 
\end{prop}

As we construct further oriented 4-manifolds that we will use as
 building blocks, we will continue to pay close attention to which
 3-manifolds form their oriented boundaries.  This will make it easier
 in the proof of Theorem \ref{ctd} to see that a certain collection of
 4-manifolds fit together to form the boundary of a 5-manifold.

Meyer \cite{Me} defined  a 2-cocycle for  $\Gamma(  \Si)$,
now called the {\em signature cocycle.} 
We denote
this cocycle by $\tau$. This cocycle  does not require a choice of a lagrangian.
In fact, 
\begin{equation} 
\notag
\tau(f,g)= 
\Signature (W(f,g))
\end{equation} where $W(f,g)$ is the 4-manifold 
which fibers over a two-holed disk, with fiber $\Si$,  and whose monodromy around the two holes is given by $f$ and $g$. See Figure \ref{ss} and \cite[4.1]{A}.
Note that $W(f,g)  
\simeq
W(g,f)$, and so $\tau(f,g)=\tau(g,f).$

\begin{figure}[h]
\includegraphics[height=1.2in]{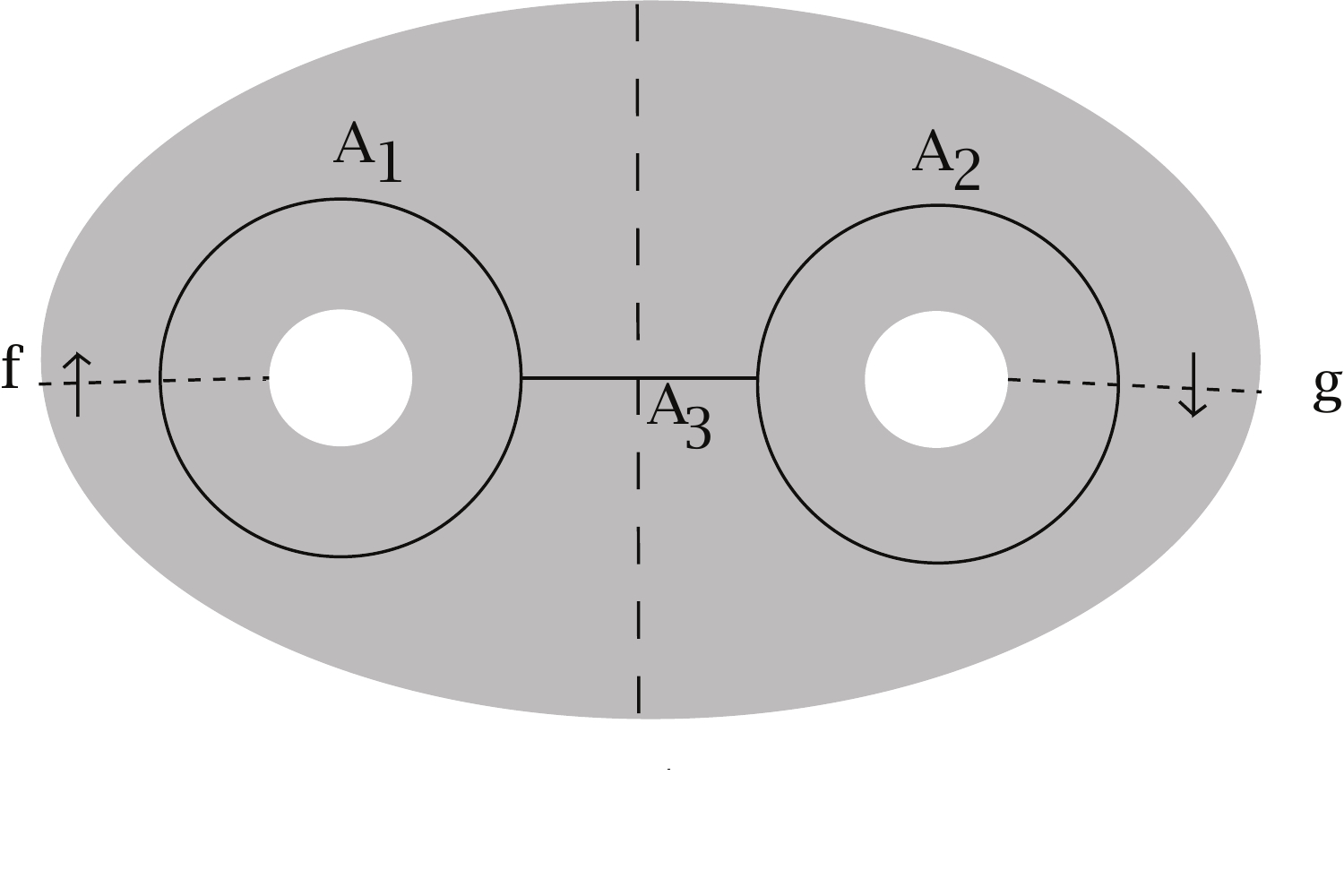}
\caption{A picture of $W(f,g)$. It is the result of cutting  a two-holed disk  times $\Si$ along the two seams $\I \times \Si$ given by the horizontal dotted lines and regluing by $f$ or $g$, as indicated. The solid lines labelled $A_1$, $A_2$ and $A_3$ indicate 2-chains that will be used  in the proof of Proposition \ref{fgint}. The vertical dotted line will also be used in the proof of this proposition.}
\label{ss}
\end{figure}

The boundary of Meyer's manifold $W(f,g)$ consists of three {\em
  mapping tori.} Here, the mapping torus of $f \in \Gamma({ 
  \Si})$, denoted  $T(f)$, is the quotient space 
$(\I \times {  \Si} )/ \sim$ 
where $\sim$ is the equivalence relation
generated by $(1,x)\sim(0,f(x))$.  One easily checks the following
\begin{prop}  As an oriented manifold, the boundary of $W(f,g)$ is 
$$\overline{T(g\circ f)}\sqcup {T(f)} \sqcup { T(g)}~.$$
\end{prop}

Turaev \cite{T2,T3} independently defined and studied a cocycle
$\varphi$ which turns out to be equal to $-\tau$ (see
Proposition~\ref{fgint} below). 
Meyer defined $\tau$ as a cocycle for $\Sp(g(\Si),\BZ)$, and Turaev
considered his cocycle as a cocycle for the symplectic group
$\Sp(g(\Si),\BR)$, but for our purposes, we just consider it as a
$2$-cocycle for $\Gamma({  \Si})$. 
Turaev
modeled the construction of $W(f,g)$ algebraically and defined  
$$\varphi(f,g)=\Signature(\star_{f,g})$$ where 
the symmetric
bilinear form $\star_{f,g}$ 
is given by the following  
\begin{prop} (Turaev) If $V$ is a rational vector space with a nonsingular skew symmetric inner product $\cdot$ and $f$ and $g$ are automorphisms which  preserve $\cdot$, then
\begin{equation} \label{tuf}  
a \star_{f,g} b =\left  (    (f-1)^{-1}a+ (g-1)^{-1} a +a \right) \cdot b 
\end{equation} defines a symmetric bilinear form on $(f-1)V \cap (g-1)V.$
\end{prop}

(Turaev's  result actually dealt  with real vector spaces.) Turaev
then proceeded to study $\varphi$ algebraically.   As
Turaev used $W(f,g)$ only for motivation or inspiration, he  did not
need to include a proof of the following proposition, which he must
have known.

\begin{prop}\label{fgint} If $f,g \in \Gamma({  \Si})$, the
  intersection form on $H_2(W(f,g))$    divided by part of its radical
  is isomorphic to {\em minus} the form 
$\star_{f,g}$ on $(f-1)H_1( \Si) \cap (g-1)H_1( \Si)$ defined in
(\ref{tuf}).  In  particular 
$$\tau(f,g)= 
\Signature(W(f,g))= - \Signature(\star_{f,g})
=-\varphi(f,g)~.$$
\end{prop}

\begin{rem}{\em This gives a topological proof that the form $\star_{f,g}$ is symmetric.}\end{rem}

\begin{proof}[Proof of Proposition~\ref{fgint}] If we cut $W(f,g)$ along the $\I \times {\Si}$ indicated by the vertical dotted line in Figure \ref{ss}, we obtain the disjoint union of $\I \times T(f)$ and $\I \times T(g).$
This gives the long exact Mayer-Vietoris sequence:
\begin{align}
& H_2(T(f)) \oplus H_2( T(g)) \rightarrow H_2(W(f,g)) \rightarrow H_1({\Si}  )\notag \rightarrow 
H_1( T(f)) \oplus H_1( T(g)) \notag
\end{align}

The image of the first arrow is 
 contained  
in the radical of the intersection
form. Thus, on the cokernel of this map, there is an induced bilinear
symmetric form whose signature is  
$\Signature(W(f,g))$.
On the other hand,
this cokernel is isomorphic to the kernel of the last arrow  which can
be identified with $(f-1)H_1( \Si) \cap (g-1)H_1(\Si).$ We only need
to see that the middle arrow (which is the Mayer-Vietoris boundary
map) sends  the intersection form on $W(f,g)$ to minus the form  $\star_{f,g}$. 

We may describe a homology class in $H_2(W(f,g))$
which maps to an element $a\in (f-1)H_1( \Si) \cap
(g-1)H_1(\Si)\subset H_1( \Si)$ as follows.  Suppose $\alpha_1$ and
$\alpha_2$ are 
closed oriented 
curves in ${\Si}$ such that  $(f-1)[\alpha_1]=a$, and 
$(g-1)[\alpha_2]=a.$ Then $\alpha_1$  sweeps out  in  $T(f)$  a cylinder $A_1$ which projects onto the solid circle labeled $A_1$ in Figure \ref{ss}. We think of this cylinder as a 2-chain with   boundary lying in the copy of ${\Si}$ lying over the point where the lines labelled $A_1$ and $A_3$ meet.  We will denote 
this copy of ${\Si}$ by ${\Si}_1$. By construction,  we have 
$$ [\partial A_1]= 
[f(\alpha_1)] - [\alpha_1]
= a \in H_1({\Si}_1)~.$$

 There is also a similar 2-chain $A_2$ in $T(g)$ with boundary  $\partial A_2$ representing $a$ in $H_1({\Si}_2)$, where ${\Si}_2$ is another copy of ${\Si}$  lying over the point where the lines labelled $A_2$ and $A_3$ meet. 
We connect the boundaries of the 2-chains $A_1$ and $-A_2$ by a
2-chain $A_3$  in the copy of $\I\times {\Si}$ joining ${\Si}_1$  and ${\Si}_2$ (lying over the arc labelled $A_3$ in the figure) so that $$\partial  A_3= \partial A_2 - \partial A_1~.$$ 
Then 
$A_1 + A_3-A_2$ 
gives a 2-cycle in $W(f,g)$ representing a  homology class which maps to $a$ under the Mayer-Vietoris boundary map. (The minus sign in front of $A_2$ is necessary from the definition
of the Mayer-Vietoris boundary map.) 
 
If we have another such 2-chain  
$B_1+B_3-B_2$
 mapping to  $[b] \in (f-1)H_1(\Si) \cap (g-1)H_1(\Si)$, but placed further inside and rotated slightly,  Figure \ref{int2}  indicates why 
\begin{align}\notag
[A_1 +A_3 - A_2] \cap [B_1 +B_3- B_2] 
&= - \,\, ((f-1)^{-1} a + a +(g-1)^{-1}a
) \cdot b\\
& = - \,\, a \star_{f,g} b~. \notag
\end{align}

The reason for the minus sign in this equation is that a point of intersection $x$ of the
two $2$-chains corresponding to a positive intersection point $p$ in
the base (with frame $(e_1,e_2)$, say) and a positive
intersection point $q$ in the fiber over $p$ (with frame $(e_3,e_4)$,
say) should be counted negatively, since
the frame $(e_1, e_3, e_2, e_4)$ at $x$ differs from the standard frame
$(e_1, e_2, e_3, e_4)$ by a transposition.
\end{proof}

\begin{figure}[h]
\includegraphics[height=1.2in]{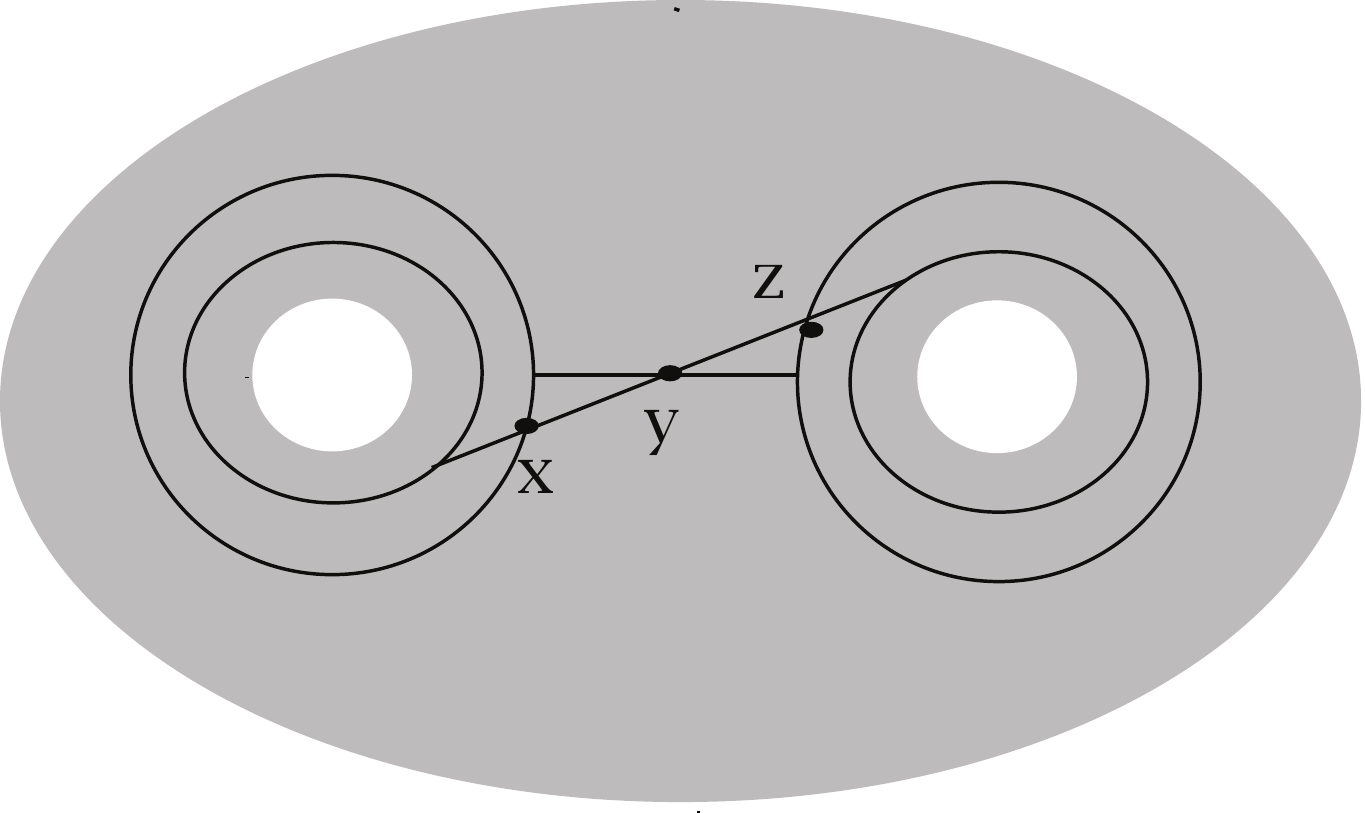}
\caption{The point  $x$ indicates $A_1 \cap B_3$ which contributes 
\hbox{$-(f-1)^{-1}(a) \cdot b$}. The point  $y$ indicates $A_3 \cap B_3$ which contributes 
$- a \cdot b$. The point  $z$ indicates $-A_2 \cap B_3$ which contributes 
\hbox{$- (g-1)^{-1}(a) \cdot b$.}}
\label{int2}
\end{figure}

We will need  \cite[Theorem 2]{T2,T3} where Turaev showed that the
cocycle $\varphi$ is a coboundary$\pmod 4$. Before stating this
result, we recall that, with $ \Gamma({  \Si})$ acting trivially on
$\BZ$, the coboundary 
 $\delta c$ 
of a 1-cochain $c:  \Gamma({  \Si}) \rightarrow \BZ$ is given by  
       $(\delta c) (g,h)= c(g)+c(h)-c(gh)$.
\begin{thm}[Turaev \cite{T2,T3}]\label{div}
The $1$-cochain $k$ on $ \Gamma({  \Si})$ which assigns to $f$
 
\begin{equation}\label{def-k}
k(f)= \dim\left( (f-1)H_1(\Si) \right) + 
 \sgn[\det(\star_f)] 
-1
\end{equation}
has  coboundary $\delta k$ satisfying
\[ \delta k \equiv \varphi \pmod{4} .\]
\end{thm}

\begin{rem}{\em 
The reason that we gave a proof of Proposition \ref{fgint} is that Turaev
proves Theorem
  \ref{div} for the cocycle  $\varphi$ given by
 $\Signature(\star_{f,g})$, 
and we will use the cocycle $\tau$  described by 
$\Signature(W(f,g))$.
So we need to know how exactly they are related.}
\end{rem}

Walker \cite[p.~124]{W} defines a 1-cochain $j_\lambda$ on $
\Gamma(  \Si)$ which assigns to $f$  the signature of the
4-manifold $J_\la(f)$ obtained by gluing $\I \times \BH$ along $\I
\times \Si$ in the boundary of $\I\times T(f)$ as indicated in Figure 
\ref{j}.\footnote{Walker actually draws the arrow for $f$ in the other direction, and this has the effect that his $j$ is minus our $j$.
Similarly Walker's $d(f,g)$ is minus our $\tau(f,g).$}
Here, as above, $\lambda$ is the kernel of $H_1(\Si) \rightarrow
H_1(\BH)$.  Note that the boundary of $J_\la(f)$ is 
$\overline{\Heeg(f)} \sqcup {T(f)}.$

\begin{figure}[h]
\includegraphics[height=1.2in]{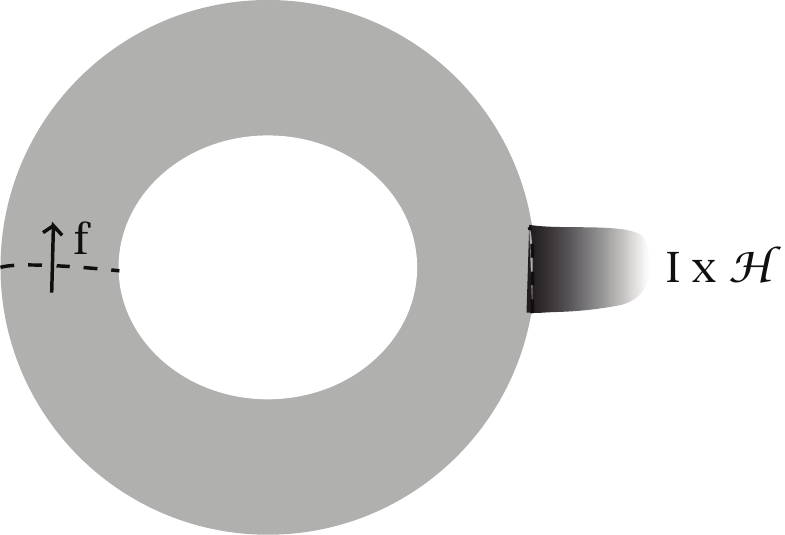}
\caption{A thickened  handlebody attached to a thickened mapping torus of $f$. Its signature is $j_\lambda(f)$. }
\label{j}
\end{figure} 

We have a formula for $j_\la(f)$ which is similar to  Turaev's formula
for $\varphi(f,g)$. Recall that $\star_{f,\lambda}$ is the symmetric bilinear form on  $\lambda
     \cap (f-1)H_1(\Si)$ obtained as the restriction of the
     non-symmetric form
     $\star_f$.

\begin{prop} \label{Walkmfd} One has 
\begin{equation}
\label{def-j} 
\Signature (J_\lambda(f))
= j_\lambda(f)= 
 - \Signature(\star_{f,\lambda})
\end{equation}
\end{prop}

\begin{proof}  We have a Mayer-Vietoris sequence:
\begin{align}
& H_2(T(f)) \oplus H_2( \BH) \rightarrow H_2(J_\lambda(f)) \rightarrow H_1( \Si) \notag  \rightarrow H_1( T(f) ) \oplus H_1(\BH \notag) 
\end{align}

The image of the first arrow is 
contained 
in the radical of the intersection
form. Thus, on the cokernel of this map, there is an induced bilinear
symmetric form whose signature is  
$\Signature(J_\lambda(f))$.
On the other
hand, this cokernel is isomorphic to the kernel of the last arrow
which can be identified with $\la \cap (f-1)H_1(  \Si).$ We  need to
show that the middle arrow sends the intersection form on $J_\la(f)$
to minus the form  $\star_{f,\la}$. 

The proof is very much like the proof  of Proposition \ref{fgint}.  
We describe a homology class in $H_2(J_\la(f))$
which maps to an element $a\in \la \cap (f-1)H_1( \Si).$  Suppose
$\alpha$ is 
an oriented curve
in $\Si$ such that
$(f-1)[\alpha]=a$. Then $\alpha$  sweeps out in  $T(f)$  a $2$-chain $A_1$ in $T(f)$ with boundary representing $a$ on $\Si_1$, a copy of $\Si$.  
Moreover there is a 2-chain $A_2$ in $\BH$ with boundary representing $a$. Then $A_1 - A_2$ gives a 2-cycle in $J_\la(f)$ representing a class which maps to $a$ under the Mayer-Vietoris boundary map. 

If we have another 2-chain $B_1 - B_2$ mapping to a  class $[b] \in
\la \cap (f-1)H_1(  \Si) $ but placed further inside and rotated slightly,  Figure \ref{j2}  indicates why $$[A_1 - A_2] \cap [B_1 - B_2]
= - a \star_{f,\la} b~.$$
\end{proof}
 
\begin{figure}[h]
\includegraphics[height=1.2in]{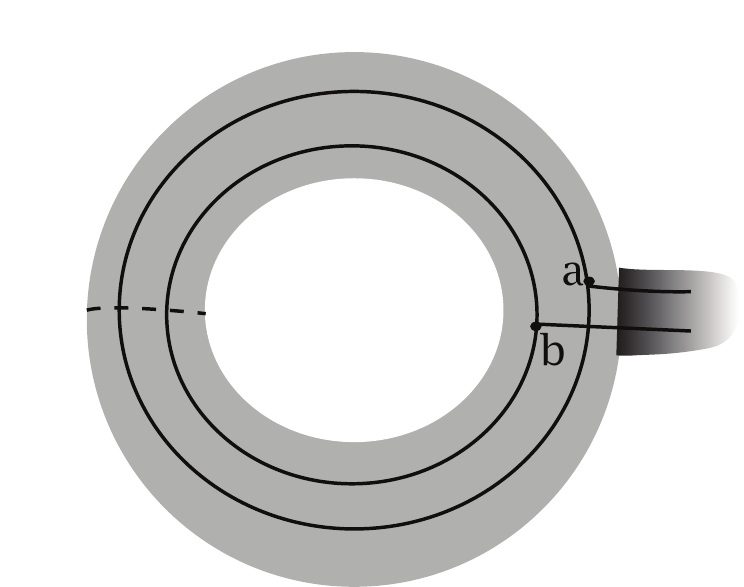}
\caption{The intersection of 2-cycles is given by $ - (f-1)^{-1} (a)\cdot b$. }
\label{j2}
\end{figure}

In the following Theorem and proof, we adapt an argument of Walker \cite[pp 123-125]{W} 
to our definitions.
Let $[m_\la]$  and $[\tau]$ represent the cohomology classes in $H^2(\Gamma({   \Si}); \BZ)$ represented by the cocycles $m_\la$ and $\tau$.
\begin{thm}[Walker] 
\label{ctd} We have that $\delta(j_\la)= \tau+m_\la$ . Thus $[m_\la]=- [\tau]$. \end{thm}
\begin{proof}  Form a 5-manifold 
 $X$   
by attaching  $D^2 \times \BH $ to $\I \times W(f,g)$ along 
$R \subset \{1\}\times W(f,g)$, where $R\approx D^2\times \Si$ is represented by the darkest shaded region  in Figure~\ref{x}.
The boundary 
 $\partial X$  
 is the union of copies  of the oriented manifolds 
$\overline{W(f,g)}$, 
$U(\BH_{g\circ f}, \BH_g, \BH)$, 
$\overline{J_\la(g \circ f)}$, ${J_\la(g)}$, ${J_\la(f)}$. To see that the orientations are as stated, note that:
\begin{align}
\partial(\overline{W(f,g)})&= T(g \circ f) \sqcup \overline{ T(g)}   \sqcup \overline{ T (f)}\notag \\ 
\partial(U(\BH_{g\circ f}, \BH_g, \BH))
&= 
\overline{\Heeg(g \circ f)} \sqcup { \Heeg(g)}   \sqcup  \Heeg (f)
\notag \\
\partial(\overline{J_\la(g \circ f)})&= \overline{T(g \circ f)} \sqcup   
{\Heeg(g\circ f)} 
\notag \\
\partial({J_\la(g)})&= {T(g)} \sqcup 
 \overline{\Heeg(g)}  
 \notag \\
\partial({J_\la(f)})&= {T(f)} \sqcup  
 \overline{\Heeg(f)} 
. \notag 
\end{align}

 These 4-manifolds  are glued along closed 3-manifolds. By Novikov
 additivity, the signature of $\partial X$ is the sum of the
 signatures of the pieces.  As the signature of a 4-manifold which is
 the boundary of a 5-manifold 
is zero, we have that: 
\[-\tau(g,f)-m_\la(g,f)-j_\la(g\circ f)+j_\la(g)+j_\la(f)=0.\]
\end{proof}

\begin{figure}[h]
\includegraphics[height=1.2in]{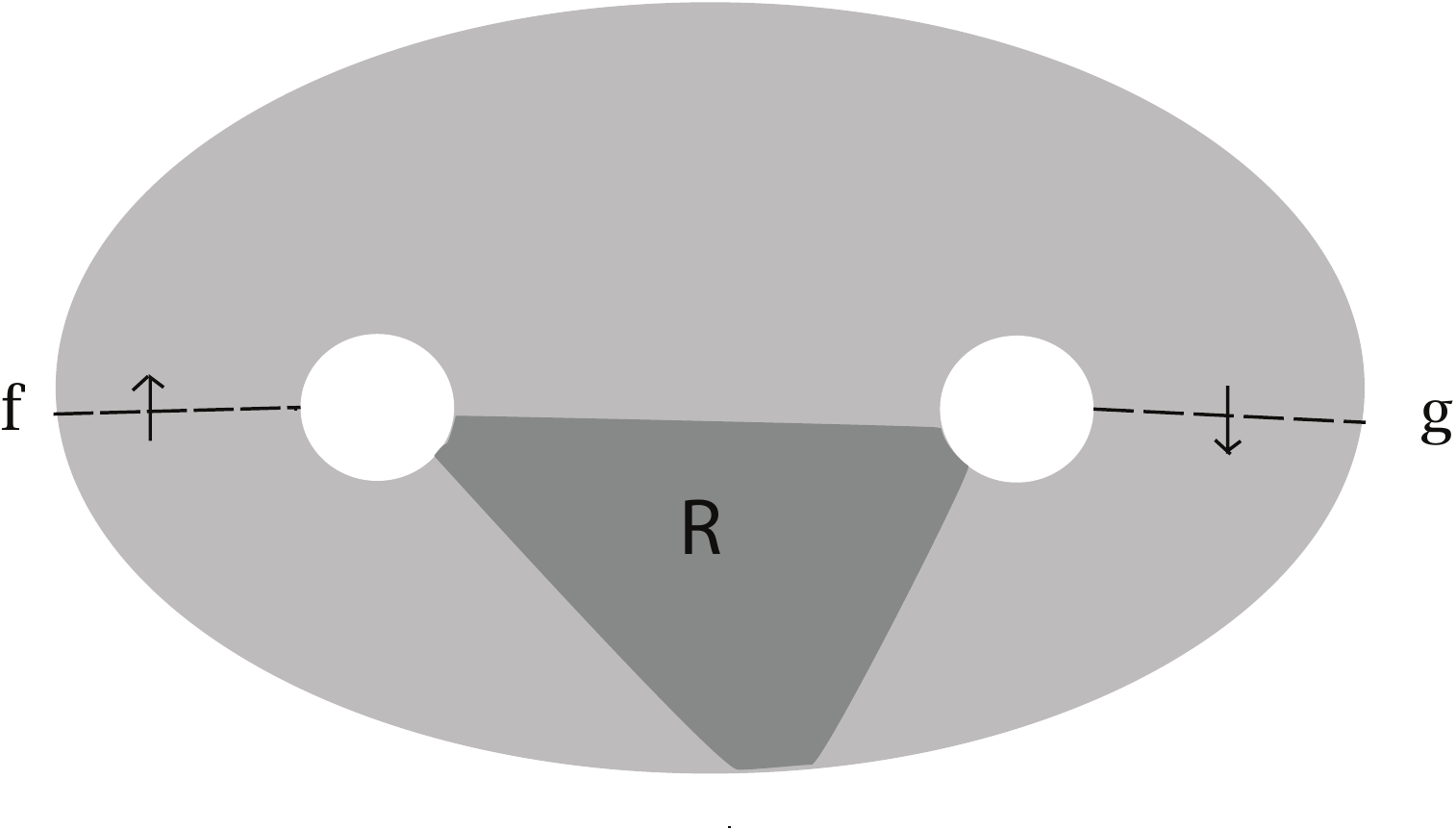}
\caption{$D^2 \times \BH$ is attached  along  a copy of $D^2 \times  \Si $ indicated by the dark shaded region $R$.}
\label{x}
\end{figure} 

We are now ready to give the

\begin{proof}[Proof of Theorem \ref{hom}]  
 Consider the subset of $\widetilde\Gamma(\Si)$ consisting of the
 $C(f,n)$ where $n\equiv
n_\la(f) \pmod 4$. We must show that this subset is a subgroup of
$\widetilde\Gamma(\Si)$. We have ${C}(g,0) \circ  {C}( f,0)=  
  {C}(g \circ f, m_\la(g,f))$ and therefore
\begin{equation}\notag 
  {C}(g,n_\la(g)) \circ  {C}( f,n_\la(f))=  
  {C}(g \circ f, n_\la(gf) +(m_\la+\delta n_\la)(g,f)).
\end{equation}  
Thus it is enough to show that 
\begin{equation}\label{m4} m_\la+\delta n_\la \equiv 0 \pmod 4~.
\end{equation}  Using (\ref{def-k}) and (\ref{def-j}), the definition
of $n_\la$ can be written as 
\begin{equation} n_\la=-j_\la-k~. \notag
\end{equation} By  Walker's  
theorem~\ref{ctd} and Turaev's 
theorem~\ref{div}, it follows 
that $$ \delta n_\la\equiv -\tau -m_\la - \varphi  \pmod 4~.$$
This implies (\ref{m4}) since  $\tau=-\varphi$ by
Proposition~\ref{fgint}. This completes the proof.

\end{proof}

\section{Surfaces with boundary}
\label{sec6}

To simplify the exposition, we delayed the discussion of surfaces with
boundary. However, all the preceding results  
hold
for the mapping
class group of a surface with boundary, modulo the following
modifications.

An extended surface with boundary is a compact oriented surface $\Si$
together with a choice of lagrangian $\la$ in
$H_1(\widehat\Si; \BQ)$, where 
$\widehat \Si$
is the closed surface obtained
from $\Si$ by attaching a  
disk to each boundary component. As
in the case without boundary, we denote the
ordinary mapping class group of $\Si$ by
$\Gamma( {\Si})$. It is
the group of orientation preserving diffeomorphisms of $\widehat\Si$ which
are the identity on the attached disks, modulo isotopies which are
again the
identity on the attached disks. 

We define the extended mapping class group $\widetilde \Gamma(\Si)$ 
to be the group of pairs $(f,n) \in \Gamma( \Si)
\times \BZ$ with multiplication 
\begin{equation} 
\notag
(g,m) \circ ( f,n)=(g \circ f, m+n + \mu_{\widehat\Si}(  \lambda,g _*
  \lambda, (g \circ f)_* \lambda))
\end{equation}

If $\Si$ has no boundary, this is equivalent to the definition of
$\widetilde \Gamma(\Si)$  in terms of mapping cylinders (see Remark~\ref{comp4}).  
If $\Si$ has boundary, we can again think of
$(f,n)$ as represented by the mapping cylinder 
$C(f,n)$ viewed as a cobordism from $\widehat\Si$ to itself.  But the
notion of equivalence of cobordisms has to be modified
appropriately so that $C(f,n)$ determines $(f,n)$. 

The groups $\widetilde \Gamma(\Si)^+$ and $\widetilde
\Gamma(\Si)^{++}$ are now defined exactly as in the closed case, and
Theorem~\ref{new4.1} 
continues to hold as stated. 
But
notice that although the curves representing Dehn
twists will avoid the attached disks,  we must, of course, use the closed surface $\widehat\Si$ (well-placed in
$S^3$ with
respect to the lagrangian $\lambda$) to construct the framed link  
$L^0_\la(\mathfrak w)$
associated to a word $\mathfrak w$ in Dehn twists. As for the algebraic description of $\widetilde
\Gamma(\Si)^{++}$,  
Theorem~\ref{hom} and Corollary~\ref{hom-ind} continue 
to hold except that in
the statement of Theorem~\ref{hom}, 
we
must replace the homology group
$H_1(\Si)$ by $H_1(\widehat\Si)$.  

The reason why all results in
Sections~\ref{sec3} --~\ref{sec5b}  go through for surfaces with
boundary is that all the extensions and cochains of the mapping class group
$\Gamma( \Si)$ we consider are pull-backs from the corresponding
extensions and cochains of $\Gamma( {\widehat\Si})$.

\section{Universal central extension}
\label{sec7}
In this section, let $\Si=\Si_{g,r}$ denote a connected compact oriented
surface of genus $g$ with $r$ boundary components. We denote the
ordinary mapping class group of $\Si_{g,r}$ by $\Gamma_{g,r}$. We also
write $\widetilde\Gamma_{g,r}$ and $\widetilde\Gamma_{g,r}^{++}$ for
the extended mapping class group $\widetilde\Gamma(\Si)$ and its index four
subgroup $\widetilde\Gamma(\Si)^{++}$. As remarked in~\ref{uni}, although
our description of these groups requires the choice of a lagrangian,
they are independent of this choice up to isomorphism.

\begin{prop} If  $g\ge 4$, then $\widetilde\Gamma_{g,r}^{++}$ is a universal central extension of $\Gamma_{g,r}$.
\end{prop}

\begin{proof}
For $g\ge 3$, $\Gamma_{g,r}$ is perfect 
(see for example \cite{FM}).
 Hence it has a universal
central extension 
\cite[p.~96]{B}. 
This is an extension by
$H_2(\Gamma_{g,r};\BZ)$
satisfying a certain universal property.   
If $g\geq 4$, it is known that
 $H_2(\Gamma_{g,r};\BZ)\simeq\BZ$ and $H^2(\Gamma_{g,r};\BZ)\simeq\BZ$. 
See \cite{KS}, and the references therein.
Hence, if $g\ge 4$,  
a universal central extension  of $\Gamma_{g,r}$ is an extension by
$\BZ$, and 
 a central extension
of $\Gamma_{g,r}$ by $\BZ$ is a universal central extension if
and only if its cohomology class is a
generator of $H^2(\Gamma_{g,r};\BZ)$.

 Meyer \cite{Me} showed that if $g\ge 3$, the cohomology class $[\tau] 
\in H^2(\Gamma_{g,0};\BZ)$ defines a map  $H_2(\Gamma_{g,0};\BZ)\rightarrow
\BZ$ whose image is $4\BZ$. This implies that $[\tau]/4$ is a
generator of $H^2(\Gamma_{g,0};\BZ)$ if $g\ge 4$. Since
$[\tau]=-[m_\la]$ and the extension
$\widetilde\Gamma_{g,0}^{++}$ is classified by $[m_\la]/4$, this shows
that $\widetilde\Gamma_{g,0}^{++}$ is a universal central extension
if $g\ge 4$. Finally, the same is true for
$\widetilde\Gamma_{g,r}^{++}$ if $r>0$, since the extension $\widetilde\Gamma_{g,r}^{++}
\rightarrow \Gamma_{g,r}$ is a pullback of the extension $\widetilde\Gamma_{g,0}^{++}
\rightarrow \Gamma_{g,0}$ and the natural map $H_2(\Gamma_{g,r};\BZ)
\rightarrow H_2(\Gamma_{g,0};\BZ)$ is an isomorphism in our situation.
\end{proof}

\section{Applications to TQFT}
\label{sec8}

A Topological Quantum Field Theory (TQFT) in the sense of Atiyah and
Segal includes in particular representations of centrally
extended mapping
class groups of surfaces. The fact that one needs to consider central
extensions is sometimes called the  `framing anomaly' of the TQFTs we
are interested in. There are essentially four ways to describe
the central extension in the literature: Atiyah's description
\cite{A} using 
$2$-framings 
and the signature cocycle, 
 Walker's description \cite{W} using
integral weights, lagrangians, and 
Maslov indices, as in Definition~\ref{comp3}, the description using
$p_1$-structures given in \cite{BHMV2} (see also Gervais \cite{Ge}),
and the description in \cite{MR} {\em via} an explicit computation of
the projective factors arising in the TQFT-representations of the
mapping class group. 

In \cite{BHMV2}, a version of the Reshetikhin-Turaev
$SU(2)$- and $SO(3)$-TQFT was constructed using the skein theory of
the Kauffman
bracket. In this section, we consider the TQFT constructed in the same
way as in \cite{BHMV2}, but with integral weights and lagrangians in
place of the $p_1$-structures. This variant of the TQFT constructed in
\cite{BHMV2} has been described and used in \cite{G,GMW, GM, GM1}.
Our
aim  here 
is to show how to use the techniques of 
Section~\ref{newsec3}
to describe the representations of the
extended mapping class group $\widetilde\Gamma(\Si)$ arising in this
TQFT explicitly, and to do some computations with these
representations which are used in \cite{GM1}
and in work in progress.

A TQFT is a
functor on a certain cobordism category with values in the category of
vector
spaces, or, more generally, modules over a commutative ring. The
cobordism category we use is an enhancement of the extended cobordism
category $\cC$ described in Section~\ref{sec3}. The enhancement
consists in allowing surfaces to contain (possibly empty) collections
of colored banded points and  3-manifolds to contain a (possibly
empty)  colored banded trivalent graph which meets the boundary in the
banded points of the boundary surfaces. As in \cite{BHMV2}, a banded
point is an oriented arc through the point. A banded trivalent graph
is a trivalent
graph together with an oriented surface which deformation retracts to
the graph. The colors are from a certain finite palette which depend on
the specific TQFT under consideration. We refer to extended surfaces
and $3$-manifolds which are enhanced in this way simply as extended
surfaces and extended $3$-manifolds.

The TQFT's we consider are indexed by an integer $p\ge 3$ and denoted
$(Z_p, V_p)$.  The notation is such that to an extended 
       surface $\Si$, there is
associated a $k_p$-module
$V_p(\Si)$, and to an extended 
       cobordism $M:\Si  \rightsquigarrow \Si'$, there is 
associated a $k_p$-linear map $$Z_p(M): V_p(\Si) \rightarrow
 V_p(\Si')~,$$ where $k_p$ denotes the ring of coefficients. The module $V_p(\emptyset)$ is canonically identified with the ground
ring $k_p$. Although our modules are not vector spaces, it is
customary in TQFT to call their elements vectors.  If $M:\emptyset  \rightsquigarrow
 \Si$, we simply write $Z_p(M)$ for the vector $Z_p(M)(1) \in V_p(\Si).$ In
 \cite{G,GM}, this vector is denoted by $\left[M\right]_p$. 

We take the ring of coefficients to be $k_p=\BZ[{\frac 1
 p},A, \kappa]$, where $A$ is a primitive $2p$-th root of unity, and $\kappa$
 is a square root of $A^{-6-p(p+1)/2}.$  Increasing the weight of an extended $3$-manifold
 $M$ by
 one multiplies the vector $Z_p(M)$ by $\kappa$. (Here we depart from the
 notation of \cite{BHMV2} whose $\kappa$ is a further third root
 of our $\kappa$.) The palette of allowed colors is 
 $\{0, \ldots, k\}$, if $p=2k+4$ is even, and $\{0, \ldots, p-2\}$, 
if $p\geq 3$ is odd.  Moreover, 
the colorings of the trivalent graphs must be $p$-admissible
\cite[p. 905]{BHMV2}.  
If $p=2k+4$ then $(Z_p, V_p)$ is a variant of the
$SU(2)$-theory at level $k$, while for odd $p$ it is called an
$SO(3)$-theory.

A fundamental ingredient in the construction of \cite{BHMV2} is the
surgery axiom which allows one to replace surgery along a banded knot with
cabling that knot with a certain skein element  $\omega$ in the solid
torus. Here, a skein element in a $3$-manifold is a linear combination of
banded links (or, more generally, colored banded graphs). In
\cite{BHMV2} the relevant notion of surgery was $p_1$-surgery. Here is a
formulation of the surgery axiom in our present context.

Let $\omega$ denote the skein element in 
solid torus 
$\overline{S^1} \times D^2$
described in \cite{BHMV2}. The boundary of $\overline{S^1} \times D^2$
is the torus $S^1 \times S^1$, which we denote by
$\mathcal{T}$. It is also the boundary of $D^2\times S^1$. 
As in Section~\ref{sec2}, we 
make
$\overline{S^1} \times D^2$ and $D^2\times S^1$ into 
extended manifolds by
giving both of them weight zero. If $\mathcal{T}$ is made into an
extended surface by equipping it  with
some lagrangian $\lambda({\mathcal T})$, then the pair $(\overline{S^1 }\times D^2,\omega)$
defines a vector $Z_p(\overline{S^1 }\times D^2,\omega)$ in
$V_p({\mathcal T})$.

\begin{lem}[Surgery Axiom] Assume $\lambda({\mathcal T})$ is the
  lagrangian 
generated by the homology class of
the meridian ${\text pt} \times S^1$ 
of $\overline{S^1 }\times D^2$. 
Then in $V_p({\mathcal T})$ one
has $$Z_p(\overline{S^1 }\times D^2,\omega)= Z_p(D^2 \times S^1)~.$$
\end{lem}
The proof of the Surgery Axiom in our current context of extended
manifolds is completely
analogous to the proof of this axiom in the original context of
\cite{BHMV2}. We omit the details.

Now let $\Si$ be 
 a connected
 extended surface, with
lagrangian $\lambda(\Si)$. Consider the extended mapping cylinder $C(f,n)\in \widetilde
\Gamma(\Si)$ where $(f,n) \in \Gamma( \Si)
\times \BZ$. Let 
\begin{equation}\label{defrep}
\rho_p(f,n)=Z_p(C(f,n))~.
\end{equation}

This defines a representation $\rho_p$ of $\widetilde
\Gamma(\Si)$ on $V_p(\Si)$. This representation can be described in
very concrete terms,  as
follows. Let 
 $\mathfrak w=\prod_{i=1}^N
{\alpha_i}^{\varepsilon_i}$ be a word so that $D(\mathfrak w)=f$. Let
$L(\mathfrak w)\subset  \I \times \Si $ be the framed link considered
in 
 Section~\ref{newsec3}.
 Let $s({\mathfrak w})$ be the skein element in
$\I\times \Si$ obtained by cabling every component of 
this framed link 
with $\omega$. We consider $\I\times \Si$ as an extended manifold by
giving it weight zero.

\begin{thm}\label{8.2} One has that
\begin{equation}\label{form1}
\rho_p(f,n)= 
 \kappa^{n-n_\lambda^0(\mathfrak w)}
  Z_p(\I \times \Si,s({\mathfrak w}))~.
\end{equation}
\end{thm}

Here, 
$\la$ is the given lagrangian $\lambda(\Si)$,  and 
 $n_\lambda^0(\mathfrak w)=\sigma(L^0_\la(\mathfrak  w))$, the
signature of the 
linking matrix of the 
framed link
$L^0_\la (\mathfrak w)$ (see 
Theorem~\ref{new4.1}).

\begin{proof} 
As explained in 
the proof of 
Theorem~\ref{new4.1},
extended surgery along $L(\mathfrak w)$ on the identity
  mapping cylinder $C(\Id_\Si,0)$ gives $C(f,n_0)$ where 
       $n_0=\sigma(L^0_\la(\mathfrak  w))$.
Therefore the  surgery axiom implies that 
$$Z_p(\I \times \Si,s({\mathfrak w}))= Z_p(C(f,n_0))=\rho_p(f,n_0)~.$$
This differs from $\rho_p(f,n)$ by the factor $\rho_p(W)^{n-n_0}$
where $W=C(\Id_\Si,1)$ is
  the generator of the center of $\widetilde
\Gamma(\Si)$. But $W$ acts as multiplication by $\kappa$ on
$V_p(\Si)$.  This proves the result. 
\end{proof} 

\begin{rem} {\em 
If the surface $\Si$ is connected, then 
the module $V_p(\Si)$ can be presented as
    a quotient of the skein module of a handlebody $\BH$ with boundary
    $\Si$. In other words, the natural map which sends a skein element
    $x$ in $\BH$ to the vector $Z_p(\BH,x)$ in $V_p(\Si)$, is onto. This follows from the surgery axiom as in 
    \cite[Proposition~1.9]{BHMV2}. We remark that we can choose $\BH$ 
       arbitrarily here;  
in particular, we do {\em not} need to require that the given lagrangian
    $\lambda(\Si)$ be the kernel of $H_1(\Si)
    \rightarrow H_1(\BH)$. The endomorphism $Z_p(\I \times
    \Si,s({\mathfrak w}))$ lifts to an endomorphism of the skein
    module of the handlebody $\BH$. This endomorphism can then be computed
    skein-theoretically using recoupling theory \cite{KL,MV}. Note that no further powers of $\kappa$ are
    introduced when gluing 
    $\BH$ to $(\I \times \Si,s({\mathfrak w}))$, 
    because in the Maslov index
    computation, two of the three lagrangians are the
    same.\footnote{This is true even though there may well be a Maslov index
    contribution when gluing $C(f,n)$ to $\BH$. Here one sees the
    strength of the surgery axiom.}  Thus the expression (\ref{form1})  in Theorem~\ref{8.2} gives a completely explicit
    description of $\rho_p(f,n)$. In particular, it explains how
    $\rho_p(f,n)$ depends on the lagrangian $\la$. 
  }\end{rem}

Here is an example showing how to use Theorem~\ref{8.2} to identify
specific lifts of mapping classes to the extended mapping class group.
Let ${\mathcal T}_c$ denote a torus equipped with
one banded point colored $2c$.  Assume ${\mathcal T}_c$ is presented as
the boundary of a solid torus which we will denote by $\BH$. Let $m$ and $\ell$ be 
simple closed  
curves on ${\mathcal T}_c$ which avoid the banded point, and such that $m$ is
a  meridian of $\BH$
and $\ell$ is a longitude. Mapping classes of ${\mathcal T}_c$
must preserve the banded point, so that the ordinary mapping class
group $\Gamma( {{\mathcal T}_c})$ of 
${\mathcal T}_c$  is the mapping class group of the one-holed torus
obtained from ${\mathcal T}_c$ by removing an open disk neighborhood
of the banded point. This group 
is generated by the Dehn twists $D(m)$ and $D(\ell)$. They
satisfy 
\begin{equation}
\notag
D(m)D(\ell)D(m)=D(\ell)D(m)D(\ell)~,
\end{equation}
 and this is the only
relation in a presentation of $\Gamma( {{\mathcal T}_c})$ in
terms of these generators.

In \cite{GM1}, we represented certain lifts of $D(m)$ and $D(\ell)$ to
the extended mapping class group $\widetilde\Gamma({\mathcal T}_c)$ by
certain 
automorphisms $t$ and $t^\star$ 
of $V_p({\mathcal
  T}_c)$. Using Theorem~\ref{8.2}, we can identify exactly which lifts
these are by computing their weights as extended cobordisms. Of course, for this to make sense we need to choose a lagrangian $\la$
for  ${\mathcal T}_c$. We choose $\la$ to be the lagrangian given by
$\BH$. Thus $[m]\in \la$ but $[\ell] \not\in \la$. 
\begin{prop} As automorphisms of $V_p({\mathcal
  T}_c)$, one has that 
\begin{align}
\label{tfo} t\,\,&=\rho_p(D(m),0)\\
\label{tsfo} t^\star&= \rho_p(D(\ell), 1)
\end{align}
\end{prop}

\begin{proof} The automorphisms $t$ and $t^\star$ were defined skein-theoretically in
\cite{GM1}. We briefly review the
definition. The module $V_p({\mathcal
  T}_c)$  is a quotient of the `relative' skein module of $\BH$, where
the word `relative' indicates that the skein elements are linear
combinations of banded trivalent graphs in $\BH$ which nicely meet the banded point
colored $2c$ on the boundary of $\BH$. Let $\omega_+$ denote  $\kappa$
times the skein element
in the solid torus 
obtained by giving $\omega$ a full negative twist. (An explicit
formula for  $\omega_+$,  derived from \cite{BHMV1}, is given in
\cite{GM1}.) Then $t$ is the self-map of $V_p({\mathcal
  T}_c)$ which sends a skein element $x$ to $x$ union
$\omega_+$ placed on the zero-framed meridian pushed slightly into the
interior. Another definition of $t$ is as the self-map of $V_p({\mathcal
  T}_c)$ induced by a full positive twist of the solid torus $\BH$. See
Figure~\ref{omtwist}.
\begin{figure}[h]
\includegraphics[width=1.8in]{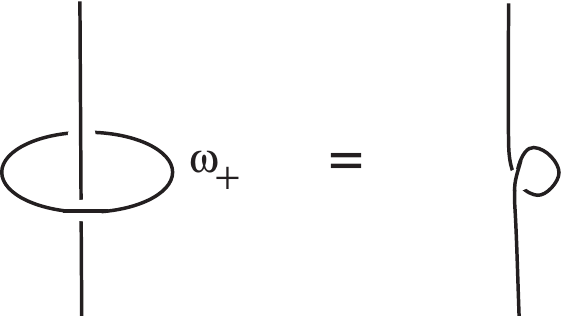}
\caption{A picture for $t$. Encircling a strand with $\omega_+$ has the same
effect in TQFT as giving that strand a positive twist.} \label{omtwist}
\end{figure} The map $t^\star$ is defined similarly 
(to the first description of $t$) 
except that we use the
zero-framed longitude in place of the zero-framed meridian. In order
to make contact with Theorem~\ref{8.2}, we denote by $m_0$ and
$\ell_0$ the meridian and longitude sitting on ${\frac 1 2} \times
{\mathcal T}_c \subset \I\times {\mathcal T}_c$, with zero framing
relative to the surface. Then the definitions of $t$ and 
$t^\star$ 
can
be reformulated as follows:  
$$t=Z_p(\I\times {\mathcal T}_c,\ \text{$m_0$ cabled by $\omega_+$})$$ 
$$t^\star=Z_p(\I\times {\mathcal T}_c,\ \text{$\ell_0$ cabled by
  $\omega_+$})$$ 
Now [$m_0$ cabled by $\omega_+$] is the same as $\kappa$ times [$m_-$
cabled by $\omega$], where $m_-$ is like $m_0$ but with $-1$ framing
relative to the surface, as in Lemma~\ref{1twist}. Thus we have a
situation like on the 
right hand side
of (\ref{form1}) in Theorem~\ref{8.2},
and formula (\ref{tfo}) for 
$t$ follows from this by a signature computation. Formula (\ref{tsfo})
follows similarly.
\end{proof}

\begin{rem}{\em  The following proof of (\ref{tfo}) and
    (\ref{tsfo}) directly from the surgery axiom may be instructive.  Since [$m_0$ cabled by $\omega_+$] is the same as $\kappa$ times [$m_-$
cabled by $\omega$], the surgery axiom gives $$t=\kappa Z_p(C(m))$$
where $C(m)$ is extended surgery along $m_-$ on
$\I\times {\mathcal T}_c$ (see Lemma~\ref{1twist}). By
Lemma~\ref{1twist}, since $[m]\in\la$, we have
$C(m)=C(D(m),-1)$. Hence $$t=\kappa
Z_p(C(D(m),-1))= Z_p(C(D(m),0))=\rho_p(D(m),0)~.$$ 
We similarly have $$t^\star=\kappa Z_p(C(\ell))$$ but this
time $[\ell]\not\in\la$,  so Lemma~\ref{1twist} gives  $C(\ell)=C(D(\ell),0)$ and hence $$t^\star=\kappa
Z_p(C(D(\ell),0))= Z_p(C(D(\ell),1))=\rho_p(D(\ell),1)~.$$ Thus, the
reason that the weights come out differently for $t$ than for $t^\star$ is
that $[m]\in \la$ but $[\ell] \not\in \la$.
}\end{rem}

\begin{rem}\label{8.6}{\em More generally, let $\alpha$ be a simple closed curve
    on $\Si$ and define $\alpha_0$ and $\alpha_-$ as above. Consider  
$$W(\alpha)= W\circ C(\alpha) \in \widetilde\Gamma(\Si)^{++}$$ as defined
    in Section~\ref{sec3}. We have 
\begin{align}\notag
\rho_p(W(\alpha))&= \kappa \rho_p(C(\alpha))= \kappa Z_p(\I\times \Si,\
\text{$\alpha_-$ cabled by $\omega$})\\
&= Z_p(\I\times \Si,\
\text{$\alpha_0$ cabled by $\omega_+$})  
\end{align} 
This defines a representation of $\widetilde\Gamma(\Si)^{++}$ on
$V_p(\Si)$. It is a fact that the skein element $\omega_+$ has
coefficients in the subring of $k_p$ spanned by $A$ and $\frac 1 p$;
in other words, $\kappa$ is not needed to define this representation.
In the case $\Si={\mathcal T}_c$, we have
    $t=\rho_p(W(m))$ and $t^\star=\rho_p(W(\ell))$. In general, we have the following skein-theoretical interpretation
    of $\rho_p(W(\alpha))$ for a simple closed curve $\alpha$. Think of $V_p(\Si)$ as
    a quotient of the skein module of a handlebody $\BH$ with boundary
    $\Si$. 
Then $\rho_p(W(\alpha))$ is the self-map of $V_p(\Si)$ which sends a skein element $x$ to $x$ union
$\omega_+$ placed on the zero-framed curve $\alpha$ pushed slightly into the
interior. We emphasize that this is true even if $\lambda(\Si)$ is not
equal to the kernel of $H_1(\Si) \rightarrow H_1(\BH)$. 
In analogy with \cite{MR}, we call $W(\alpha)$ the {\em geometric lift} of the Dehn twist $D(\alpha)$ 
to the extended mapping class group. 
}\end{rem}

The following result was stated in \cite[Remark 4.5]{GM1}. 
Let
$q=A^2$. This is a primitive $p$-th root of unity.

\begin{prop}\label{8.7} As automorphisms of $V_p({\mathcal
  T}_c)$, one has that $t t^\star t=t^\star t t^\star$ and
\begin{equation}\label{relcomp}(tt^\star)^6=q^{-6+2c(c+1)-p(p+1)/2}\Id_{V_p({\mathcal T}_c)}~.
\end{equation} 
\end{prop}
\begin{proof} The first relation, 
in a somewhat different context, 
 is well-known \cite{Ro,MR}. 
  Here is a
  proof in our context.  Since $t t^\star
  t=\rho_p(W(m\ell m))$ and $t^\star t t^\star= \rho_p(W(\ell m\ell))$, it is
  enough to show that $W(m\ell m)=W(\ell m\ell)$. This is proved as follows. We have  $D(m \ell
  m)= D(\ell m\ell)$, $e(m \ell m)=e (\ell m\ell)=3$, and $
\sigma\left(
L^0_\la
\left(m \ell m\right) \right)=\sigma\left( 
L^0_\la
\left(\ell
    m\ell\right) \right)=-2$. By 
Corollary~\ref{new52},
 it follows that
both $W(m\ell m)$ and $W(\ell m\ell)$ are equal to $C(D(m\ell m),1)$. 

For the second relation, let $\delta$ be a simple
closed curve in ${\mathcal
  T}_c$ around the banded point colored $2c$. In the mapping class
group $\Gamma( {{\mathcal T}_c})$, the Dehn twist $D(\delta)$
is equal to $\big(D(m)D(\ell)\bigr)^6$.  Let $\mathfrak u$ denote the word $(m \ell)^6 \delta^{-1}$ which is a relator.
We have $e( \mathfrak u)= 11$ and 
$\sigma\left(
L_\la
( \mathfrak u )\right)= -7$.  (N.b., it is frequently efficient to begin such
signature calculations with a simplification of the framed link using
Kirby calculus while keeping track of signature changes.) Thus 
we deduce from Corollary~\ref{new52} and Lemma~\ref{relator} that
$$W({\mathfrak u})= C(\Id_{{\mathcal
  T}_c}, 11-7)=C(\Id_{{\mathcal
  T}_c},4)~,$$ hence
$\rho_p(W(\mathfrak u))$ is multiplication by $\kappa^4$. It follows that 
$$(tt^\star)^6=\Big(\rho_p(W(m))\rho_p(W(\ell))\Big)^6 =\kappa^4
\rho_p(W(\delta))= q^{-6-p(p+1)/2}\rho_p(W(\delta))~.$$
It remains to see that $$\rho_p(W(\delta))=
q^{2c(c+1)}\Id_{V_p({\mathcal T}_c)}~.$$ In view of Remark~\ref{8.6},
this can be done by a
skein-theoretical computation. We have to compute the effect of
encircling a $2c$-colored 
strand by $\omega_+$. As shown in Figure~\ref{omtwist},
this is the same as giving that strand a full positive twist.  Recall
from \cite{BHMV1} that the twist eigenvalue is $\mu_c=
(-A)^{c(c+2)}$. Thus $\mu_{2c}=q^{2c(c+1)}$. This completes the proof.
\end{proof}

We end this section with one further technique allowing one to identify
specific lifts of mapping classes to the extended mapping class
group. This will allow us to compute how $(t t^\star)^3$ acts on
$V_p({\mathcal T}_c)$, thereby giving another proof
of (\ref{relcomp}) in which the signature computation is easier.   
Let $\Si$ be the boundary of a handlebody $\BH$, which we give
weight zero.  Recall that
every element of $V_p(\Si)$ can be written $Z_p(\BH,x)$ for some skein
element $x$ in $\BH$. 

\begin{prop} Assume $f\in \Gamma( {\Si})$ is the restriction
  of a diffeomorphism $F$ of $\BH$. Assume further that the lagrangian
  $\lambda(\Si)$ is the kernel of $H_1(\Si) \rightarrow
H_1(\BH)$. Then $\rho_p(f,0)$ sends $Z_p(\BH,x)$ to $Z_p(\BH,F(x))$.
\end{prop}

\begin{proof} 

If we glue the pair $(\BH,x)$ to  the mapping cylinder of $f$ by
identifying the boundary of $\BH$ with the source of the mapping
cylinder by the identity map, and if we forget the weights for a
moment, the result is diffeomorphic, rel.~boundary, to the pair $(\BH,F(x))$. 
Thus the proposition holds up to a power of $\kappa$ which might come
from a Maslov index contribution. But in our situation $f$ preserves
$\lambda(\Si)$, so the Maslov index contribution is zero. This completes
the proof.
\end{proof}
\begin{ex} {\em Consider again the torus ${\mathcal T}_c$ equipped with
one banded point colored $2c$.  The Dehn twist $D(\delta)$ has a
square root $\theta$ called the {\em half-twist}.
This can be roughly described as the result of giving most of 
the   
torus 
 (sitting in 3-space as the boundary of 
an
unknotted solid torus) 
 a right handed twist through an angle $\pi$ 
  around 
  an axis passing through the banded point 
 (and three other 
points on the torus)
  while holding a neighborhood of the banded point fixed.  
In the mapping class
group $\Gamma( {{\mathcal T}_c})$, one has $\theta=
(D(m)D(\ell))^3$. Now $\theta$ extends to the solid torus, so
the proposition tells us that
$\rho_p(\theta, 0)$ can be computed skein-theoretically. 
The result is that 
\begin{equation}\label{half}
\rho_p(\theta,0)= (-1)^c q^{c(c+1)}
  \Id_{V_p({\mathcal T}_c)}~.
\end{equation} 
}\end{ex}
 
This calculation 
can be done by considering the basis
       $L_{c,0}  z^n$ 
of $V_p({\mathcal T}_c)$ 
 in
  the notation of 
  \cite{GM1}. 
  This basis
consists of eigenvectors, all with the same eigenvalue. Moreover,
since $\theta$ is a square root of the Dehn twist $D(\delta)$, the
eigenvalue must be a square
  root of $\mu_{2c}=q^{2c(c+1)}$. Determining the sign 
of the square root
is, however, a
  little subtle.
One way to see the factor $(-1)^c$ in (\ref{half}) is to convince
onself by drawing some pictures that the eigenvalue is $\delta(2c;c,c)\mu_c$, where
$\delta(2c;c,c)=A^{c^2}$ is the half-twist coefficient of
\cite[Theorem 3]{MV}.

\begin{cor} As automorphisms of $V_p({\mathcal
  T}_c)$, one has $$(tt^\star)^3=A^{-6-p(p+1)/2}(-1)^c q^{c(c+1)}\Id_{V_p({\mathcal T}_c)}~.$$
\end{cor}
\begin{proof} We have that  $e((m\ell)^3)=6$, and $
  \sigma( 
L^0_\la
((m\ell)^3))=-4$.  Thus
  $$W((m\ell)^3)=C(\theta, 2)$$ by   
Corollary~\ref{new52}.
Hence
$$ (tt^\star )^3= \rho_p(W((m\ell)^3))=\rho_p(\theta,2)=\kappa^2
\rho_p(\theta,0)~, $$ which implies the result in view of
(\ref{half}).
\end{proof}

\section{Integral TQFT and representations in characteristic $p$}

In this section, we consider the $SO(3)$-TQFT $(Z_p,V_p)$ where
$p\geq 5$ is a prime. An integral refinement of this TQFT was defined
and studied in \cite{G,GM}. This gives in particular rise to
finite-dimensional representations of the ordinary mapping class group
in characteristic $p$. Our aim in this section is to explain 
the 
 role 
 played by the extensions $\widetilde
\Gamma(\Si)^+$ and $\widetilde
\Gamma(\Si)^{++} $ in this construction.

Recall $q=A^2$ is a primitive  $p$-th root of unity.
We denote the cyclotomic ring $\BZ[q]$ by
 $\BOplus$. We refer the
 reader to section 13 of \cite{GM} for the definition of the (refined)
 integral TQFT-module $\BSplus(\Si)$. It is a free $\BOplus$-module of
 finite rank. The ring $\BOplus$ is a Dedekind domain, and we sometimes refer to $\BSplus(\Si)$ as a
lattice. There is a canonical inclusion 
$$\BSplus(\Si) \hookrightarrow V_p(\Si)~.$$ We can think of this
inclusion as tensoring with $k_p$, the coefficient ring of
$V_p(\Si)$. Note that $k_p$ is obtained from $\BOplus$ by adjoining
$p^{-1}$ and $\kappa$ to it. 

Consider the action $\rho_p$ of the extended mapping class group $\widetilde
\Gamma(\Si)$ on $V_p(\Si)$ defined as in (\ref{defrep}) by 
$\rho_p(f,n)=Z_p(C(f,n)).$
Here is one of the main results of integral TQFT.

\begin{thm}[\cite{GM}] \label{9.1}
If $p \equiv 3
 \pmod{4}$, then the lattice $\BSplus(\Si)$ is preserved by $\widetilde
\Gamma(\Si)$. If $p \equiv 1
 \pmod{4}$, then $\BSplus(\Si)$ is preserved by the index two subgroup
  $\widetilde\Gamma(\Si)^{+}$ of $\widetilde
\Gamma(\Si)$. 
\end{thm}

\begin{rem}{\em This result is stated in \cite[Section~13]{GM}. The
    reason that we need to restrict to $\widetilde\Gamma(\Si)^{+}$ if $p \equiv 1
 \pmod{4}$ is that in this case $\kappa=\rho_p(\Id_\Si, 1)$ does not
 lie in $\BOplus$. (But for $p \equiv 3
 \pmod{4}$, one has  $\kappa \in\BOplus$.) In
 \cite{GM}, we therefore mainly considered the slightly bigger coefficient
 ring $\BO=\BOplus[\kappa]$ and the lattice $\BS(\Si)= \BSplus(\Si)
 \otimes \BO$. (If $p \equiv 3
 \pmod{4}$, one has $\BO=\BOplus$ and $\BS(\Si)= \BSplus(\Si)$.) The
 lattice $\BS(\Si)$ is always preserved by the extended mapping
 class group $\widetilde\Gamma(\Si)$. 
}\end{rem}

Let $h$ denote $1-\zeta_p$; this is a prime in $\BOplus$. For every $N\geq 0$,
we may consider $${\mathcal{S}}^+_{p,N}(\Si)= \BSplus(\Si)\slash h^{N+1}
\BSplus(\Si)~,$$ which is a free module over the 
 quotient
 ring $\BOplus
\slash h^{N+1} \BOplus$. Note that for $N=0$ this ring is the finite
field $\BF_p$, so that ${\mathcal{S}}^+_{p,0}(\Si)$ is a
finite-dimensional $\BF_p$-vector space. 

\begin{de} Let $\rho_{p,N}$ be the representation on
  ${\mathcal{S}}^+_{p,N}(\Si)$ induced from $\rho_p$, where we
  restrict $\rho_p$ to $\widetilde\Gamma(\Si)^{+}$ if $p \equiv 3
 \pmod{4}$, and to $\widetilde\Gamma(\Si)^{++}$ if $p \equiv 1
 \pmod{4}$.
\end{de}

Note that in this definition, we have restricted to a further index
two subgroup with respect to the statement in Theorem~\ref{9.1}. This
is needed for the following corollary to hold.

\begin{cor}\label{modular} The representation $\rho_{p,0}$ on the $\BF_p$-vector
  space ${\mathcal{S}}^+_{p,0}(\Si)$ 
 factors through
a representation of
  the ordinary mapping class group $\Gamma( {\Si})$.
\end{cor}
\begin{proof} The generator of the kernel of
  $\widetilde\Gamma(\Si)^{+} \rightarrow \Gamma( {\Si})$ acts
  by $\kappa^2= A^{-6-p(p+1)/2}$. Since $p$ is odd and $A$ is a
  primitive $2p$-th root of unity, we have $A= - q^{(p+1)/2}$. It
  follows that $\kappa^2$ is $(-1)^{p(p+1)/2}$ times a power of
  $q$. Since $q \equiv 1 \pmod{h}$, it follows that $\kappa^2 \equiv
  (-1)^{p(p+1)/2} 
  \pmod{h}.$ Thus $\kappa^2 $ acts
  trivially on ${\mathcal{S}}^+_{p,0}(\Si)$  if $p \equiv 3
  \pmod{4}$. But if $p
  \equiv 1 \pmod{4}$, then $\kappa^2 $
  acts by $-1$ and only $\kappa^4$ acts trivially. 
\end{proof}

\begin{rem}{\em In practice, in order to compute $\rho_{p,0}(f)$ for a
    mapping class $f$, one should fix a lagrangian $\la$, 
    compute $\rho_p(f,n)$ for some  $n\equiv n_\la(f)\pmod 4$, write
    $\rho_p(f,n)$ as
    a matrix in a basis of the lattice $\BSplus(\Si)$ (see \cite{GM}),
    and reduce 
    coefficients modulo $h$. Of course, if $p \equiv 3
  \pmod{4}$, it suffices to take $n\equiv n_\la(f)\pmod 2$. Another
  way to make sure that one uses a lift of $f$ to the correct subgroup
  of the extended mapping class group is to write $f$ as a word in Dehn twists and to use the `geometric' lifts, as explained in
  Remark~\ref{8.6}. 
}\end{rem}

\begin{rem}{\em
In the case $p\equiv 1\pmod{4}$, the proof of 
\cite[14.2]{GM}
should
be amended to read 
$\widetilde\Gamma(\Si)^{++}$
instead of
`the (even) extended mapping class group'. 
  In the last sentence of  
\cite[p.837]{GM},
$\widetilde\Gamma(\Si)^{+}$ should be replaced with  $\widetilde\Gamma(\Si)^{++}$.
}\end{rem}

\begin{rem}{\em One may think of the sequence of representations  
$\rho_{p,N}$ as the $h$-adic expansion of the representation
$\rho_p$. Explicit matrices for this expansion
in the case of a one-holed torus were given in \cite{GM1}. 
Note that each $\rho_{p,N}$ factors through a finite group, 
since ${\mathcal S}^+_{p,N}(\Si)$ is
 a free module
of finite rank over $\BOplus
\slash h^{N+1} \BOplus$, which itself is finite. 
Thus the $h$-adic
expansion approximates the TQFT-representation $\rho_p$ by
representations into bigger and bigger finite groups. We believe this
$h$-adic expansion deserves further study.
}\end{rem}

\end{document}